\documentclass[12pt,draftcls,onecolumn]{IEEEtran}

\def\bsa{{\boldsymbol{a}}}
\def\bsb{{\boldsymbol{b}}}

\def\bse{{\boldsymbol{e}}}

\def\bsg{{\boldsymbol{g}}}

\def\bsq{{\boldsymbol{q}}}

\def\bsv{{\boldsymbol{v}}}

\def\bsx{{\boldsymbol{x}}}

\def\bsH{{\boldsymbol{H}}}

\def\bsK{{\boldsymbol{K}}}
\def\bsL{{\boldsymbol{L}}}

\def\bsM{{\boldsymbol{M}}}

\def\bsP{{\boldsymbol{P}}}

\usepackage{amssymb}
\usepackage{algorithm,algorithmic}

\usepackage{tikz}
\usepackage{tkz-berge}
\usepackage{lipsum}
\usetikzlibrary{calc}
\usetikzlibrary{shapes,fit}
\usepackage{verbatim}
\usepackage{setspace}
\usepackage{multirow}
\usetikzlibrary{shapes,arrows}
\usepackage[short]{optidef}

\usepackage{commath}
\usepackage{amsthm}

\usepackage{color}
\definecolor{gray}{RGB}{128,128,128}

\usepackage{cite}
\usepackage{graphicx}
\graphicspath{{figure/}}
\usepackage{epstopdf}
\usepackage{tikz}
\usetikzlibrary{arrows,shapes,backgrounds}

\usepackage{bm}

\usepackage{pifont}

\usepackage{array}
\newcolumntype{M}[1]{>{\centering\arraybackslash}m{#1}}
\newcolumntype{N}{@{}m{0pt}@{}}

\usepackage{multirow}

\usepackage{stfloats}
\usepackage{caption}
\usepackage{subcaption}

\newtheorem{theorem}{Theorem}
\newtheorem{assumption}{Assumption}

\newtheorem{lemma}{Lemma}

\newtheorem{definition}{Definition}

\DeclareMathOperator{\sign}{sign}
\DeclareMathOperator{\nullrank}{null}

\DeclareMathOperator{\col}{col}

\DeclareMathOperator{\diag}{diag}
\DeclareMathOperator{\Deg}{Deg}

\DeclareMathOperator*{\argmin}{arg\, min}

\newenvironment{proof}[1][Proof]%
  {\smallskip\par\noindent\textbf{#1\,:\ }}%
  {\hspace*{\fill} \rule{6pt}{6pt}\smallskip}
\newenvironment{proof*}[1][Proof]%
  {\smallskip\par\noindent\textbf{#1\,:\ }}%

\newlength{\fwidth}

\allowdisplaybreaks
\begin{document}
\IEEEoverridecommandlockouts
\title{ \bf Communication Compression for\\ Distributed Nonconvex Optimization}
\author{Xinlei Yi, Shengjun Zhang, Tao Yang, Tianyou Chai, and Karl H. Johansson
\thanks{This work was supported by the
	Knut and Alice Wallenberg Foundation, the  Swedish Foundation for Strategic Research, the Swedish Research Council, the National Natural Science Foundation of China under grants 62133003, 61991403, 61991404, and 61991400, and the 2020 Science and Technology Major Project
	of Liaoning Province under grant 2020JH1/10100008.}
\thanks{X. Yi and K. H. Johansson are with the Division of Decision and Control Systems, School of Electrical Engineering and Computer Science, KTH Royal Institute of Technology, and they are also affiliated with Digital Futures, 100 44, Stockholm, Sweden. {\tt\small \{xinleiy, kallej\}@kth.se}.}%
\thanks{S. Zhang is with the Department of Electrical Engineering, University of North Texas, Denton, TX 76203 USA. {\tt\small  ShengjunZhang@my.unt.edu}.}
\thanks{T. Yang and T. Chai are with the State Key Laboratory of Synthetical Automation for Process Industries, Northeastern University, 110819, Shenyang, China. {\tt\small \{yangtao,tychai\}@mail.neu.edu.cn}.}
}

\maketitle

\begin{abstract}                
	This paper considers distributed nonconvex optimization with the cost functions being distributed over agents. Noting that information compression is a key tool to reduce the heavy communication load for distributed algorithms as agents iteratively communicate with neighbors, we propose three distributed primal--dual algorithms with compressed communication. The first two algorithms are applicable to a general class of compressors with bounded relative compression error and the third algorithm is suitable for two general classes of  compressors with bounded absolute compression error. We show that the proposed distributed algorithms with compressed communication have comparable convergence properties as state-of-the-art algorithms with exact communication. Specifically, we show that they can find first-order stationary points with sublinear convergence rate $\mathcal{O}(1/T)$ when each local cost function is smooth, where $T$ is the total number of iterations, and find global optima with linear convergence rate under an additional condition that the global cost function satisfies the Polyak--{\L}ojasiewicz condition. Numerical simulations are provided to illustrate the effectiveness of the theoretical results.
	
\end{abstract}



\section{Introduction}

We consider distributed nonconvex optimization. Specifically, consider a network of $n$ agents, each of which has a private local (possibly nonconvex) cost function $f_i: \mathbb{R}^{d}\mapsto \mathbb{R}$.
The whole network aims to solve the following optimization problem
\begin{align}\label{nonconvex:eqn:xopt}
	\min_{x\in \mathbb{R}^d} f(x):=\frac{1}{n}\sum_{i=1}^nf_i(x).
\end{align}
Throughout this paper we assume each $f_i$ is smooth.
Note that each agent alone cannot solve the above optimization problem since it does not know other agents' local cost functions. 
Therefore, agents need to communicate with each other through an underlying communication network. Distributed nonconvex optimization has  wide applications, such as power allocation in wireless adhoc networks \cite{bianchi2012convergence}, distributed clustering \cite{forero2011distributed},  dictionary learning \cite{wai2015consensus}, and empirical risk minimization \cite{bottou2018optimization}.

The problem \eqref{nonconvex:eqn:xopt} has been extensively studied in the literature, e.g., \cite{hong2017prox,daneshmand2018second,hajinezhad2019perturbed,sun2019distributed,yi2021linear,
	chang2020distributed,xin2021improved},  just to name a few. Due to nonconvexity, these studies typically showed that first-order stationary points can be found at a sublinear convergence rate. For example, \cite{hong2017prox,daneshmand2018second,hajinezhad2019perturbed,sun2019distributed,yi2021linear} showed that first-order stationary points can be found with an $\mathcal{O}(1/T)$ convergence rate, where $T$ is the total number of iterations. However, the algorithms proposed in these studies require significant amount of data exchange as agents iteratively communicate with neighbors. Noting that communication bandwidth and power
are limited, it is vital to propose communication-efficient distributed algorithms. In this paper, we propose distributed algorithms with compressed communication to improve communication efficiency.

\subsection{Related Works and Motivation}
It is straightforward to combine existing distributed algorithms and communication compression directly. However, such a simple strategy does not converge to the accurate solution due to the compression error, and even may lead to divergence as the compression error would accumulate. Examples have been provided in \cite{lu2020moniqua,beznosikov2020biased} to illustrate this. Therefore, communication compression in distributed algorithms has gained considerable attention recently.

When each local cost function is convex, various distributed algorithms with compressed communication have been proposed. For example, 
\cite{alistarh2017qsgd,horvath2019stochastic} used unbiased compressors with bounded relative compression error to design distributed stochastic gradient descent (SGD) algorithms; \cite{Koloskova2019decentralized} employed biased but contractive compressors  to design a distributed SGD algorithm;  \cite{reisizadeh2019exact} and \cite{liu2021linear,kovalev2021linearly} utilized unbiased  compressors to respectively design distributed gradient descent and primal--dual algorithms; \cite{nedic2008distributed} and \cite{zhu2016quantized} made use of the standard uniform quantizer to respectively design distributed subgradient methods and alternating direction method of multipliers approaches; \cite{yuan2012distributed,doan2020convergence} and \cite{doan2020fast} respectively adopted the unbiased random quantization and the adaptive quantization to design distributed projected subgradient algorithms; \cite{yi2014quantized} and
\cite{lee2018finite,magnusson2020maintaining,kajiyama2020linear,xiong2021quantized} exploited the standard uniform quantizer with dynamic quantization level to respectively design distributed subgradient and primal--dual algorithms; and \cite{lei2020distributed} applied the standard uniform quantizer with a fixed quantization level to design a distributed gradient descent algorithm. The compressors mentioned above can be unified into three general classes. Specifically,  \cite{li2021compressed} proposed a wider class of compressors with bounded relative compression error which covers the compressors used in \cite{alistarh2017qsgd,horvath2019stochastic,Koloskova2019decentralized,
	reisizadeh2019exact,liu2021linear,kovalev2021linearly}; \cite{khirirat2020compressed} considered a general class of compressors with globally bounded absolute compression error which accommodates the compressors used in \cite{yuan2012distributed,doan2020convergence,doan2020fast}; and \cite{zhang2021innovation} studied a general class of compressors with locally bounded absolute compression error which contains the compressors used in \cite{yi2014quantized,lee2018finite,magnusson2020maintaining,kajiyama2020linear,
	xiong2021quantized,lei2020distributed}.
These studies also analyzed the convergence properties of the proposed algorithms. Especially, some of them showed that the achieved convergence rates under compressed communication are comparable to and even match those under exact communication. For instance, linear convergence was achieved in \cite{liu2021linear,kovalev2021linearly,lee2018finite,magnusson2020maintaining,kajiyama2020linear,
	xiong2021quantized,li2021compressed,zhang2021innovation} under the standard strong convexity assumption.

While various algorithms with compressed communication have been designed for distributed convex optimization,  communication compression for distributed nonconvex optimization is relatively less studied because the analysis is more challenging due to the nonconvexity. Moreover, when considering distributed nonconvex optimization,  most of existing distributed algorithms with compressed communication are SGD algorithms although different types of compressors have been used. For instance, \cite{lu2020moniqua} used the \underline{mo}dular arithmetic for commu\underline{ni}cation \underline{qua}ntization (Moniqua); \cite{tang2018communication}  used unbiased compressors with bounded relative or absolute compression error; \cite{koloskova2020decentralized,taheri2020quantized,singh2020sparq,singh2021squarm} used biased but contractive compressors; \cite{reisizadeh2019robust} used unbiased compressors with bounded absolute compression error. These studies also analyzed the convergence properties of the proposed algorithms. For instance, \cite{lu2020moniqua,tang2018communication,koloskova2020decentralized,taheri2020quantized,singh2020sparq,singh2021squarm} showed that the proposed SGD algorithms with compressed communication achieve linear speedup convergence rate $\mathcal{O}(1/\sqrt{nT})$, which is the same as that achieved by distributed SGD algorithms with exact communication. 
Observing this, one core theoretical question arises.

(Q1) Under compressed communication, can first-order stationary points be found with the well-known $\mathcal{O}(1/T)$ convergence rate?

On the other hand, noting that it has been shown in \cite{karimi2016linear,yi2021linear,tang2019distributed} that global optima of nonconvex optimization can be linearly found if the global cost function satisfies the Polyak--{\L}ojasiewicz (P--{\L}) condition, another core theoretical question arises.

(Q2) Under compressed communication, can global optima be linearly found when the global cost function satisfies the P--{\L} condition?

\subsection{Main Contributions}
In this paper, we provide positive answers to the above questions. 
More specifically, the contributions of this paper are summarized as follows.

(C1) We first use a general class of compressors with bounded relative compression error, which incorporates various commonly used compressors including unbiased compressors and biased but contractive compressors,
to design a communication-efficient distributed algorithm (Algorithm~\ref{nonconvex:algorithm-pdgd}). This algorithm only requires each agent to communicate one compressed variable with its neighbors per iteration. We show that this compressed communication algorithm has comparable convergence properties as state-of-the-art algorithms with exact communication. Specifically, we show in Theorem~\ref{nonconvex:thm-sm} that it can find a first-order stationary point with the well-known $\mathcal{O}(1/T)$ convergence rate, thus (Q1) is answered. Moreover, if the global cost function satisfies the P--{\L} condition, we show in Theorem~\ref{nonconvex:thm-ft} that it can find a global optimum with linear convergence rate, thus (Q2) is answered.

(C2) We then propose an error feedback based compressed communication algorithm (Algorithm~\ref{nonconvex:algorithm-pdgd_ef}) for biased compressors particularly. This algorithm can correct the bias induced by biased compressors under the cost that it requires each agent to communicate two compressed variables with its neighbors per iteration. We show in Theorems~\ref{nonconvex:thm-sm_ef} and \ref{nonconvex:thm-ft_ef} that this algorithm has similar convergence properties as the first algorithm, which  respectively answer (Q1) and (Q2).

(C3) We finally use two general classes of compressors with globally and locally bounded absolute compression error, which cover various commonly used compressors including unbiased compressors with bounded variance, random/adaptive/ uniform quantization, and even $1$-bit binary quantizer, to design a communication-efficient distributed algorithm (Algorithm~\ref{nonconvex:algorithm-pdgd_determin}). This algorithm also only requires each agent to communicate one compressed variable with its neighbors per iteration. When the compressors have globally bounded absolute compression error, we show in Theorems~\ref{nonconvex:thm-sm_quantization} and \ref{nonconvex:thm-ft_quantization} that this algorithm has similar convergence properties as the first algorithm, which respectively answer (Q1) and (Q2). When the compressors have locally bounded absolute compression error, we show in Theorem~\ref{nonconvex:thm-ft_determin} that this algorithm can find a global optimum with linear convergence rate if the global cost function satisfies the P--{\L} condition and the corresponding  P--{\L} constant is known a priori, which answers (Q2).

In summary, the main contribution of this paper is to propose three distributed primal--dual algorithms with compressed communication for distributed nonconvex optimization, which have comparable convergence properties as state-of-the-art algorithms with exact communication. This is a significant theoretical development and to the best of our knowledge, it is the first time to achieve this. 

\subsection{Outline}
The rest of this paper is organized as follows. Section~\ref{nonconvex:sec-preliminary} introduces some preliminaries. Section~\ref{nonconvex:sec-proformu} presents the problem formulation. Sections~\ref{nonconvex:sec-main-dc}--\ref{nonconvex:sec-main-dc_determin} provide three communication-efficient distributed algorithms and analyze their convergence properties. Section~\ref{nonconvex:sec-simulation} gives numerical simulations. Finally, Section~\ref{nonconvex:sec-conclusion} concludes this paper. 

\noindent {\bf Notations}: $\mathbb{N}_0$ denotes the set of nonnegative integers.  $[n]$ denotes the set $\{1,\dots,n\}$ for any positive constant integer $n$.  $\|\cdot\|_p$ represents the $p$-norm for
vectors or the induced $p$-norm for matrices, and the subscript is omitted when $p=2$.
Given a differentiable function $f$, $\nabla f$ denotes its gradient.
${\bf 1}_n$ (${\bf 0}_n$) denotes the column one (zero) vector of dimension $n$.  ${\bf I}_n$ is the $n$-dimensional identity matrix. $\col(z_1,\dots,z_k)$ is the concatenated column vector of vectors $z_i\in\mathbb{R}^{d_i},~i\in[k]$. Given a vector $[x_1,\dots,x_n]^\top\in\mathbb{R}^n$, $\diag([x_1,\dots,x_n])$ is a diagonal matrix with the $i$-th diagonal element being $x_i$. The notation $A\otimes B$ denotes the Kronecker product of matrices $A$ and $B$. Given two symmetric matrices $M,N$, $M\ge N$ means that $M-N$ is positive semi-definite. $\nullrank(A)$ is the null space of matrix $A$. $\rho(\cdot)$ stands for the spectral radius for matrices and $\rho_2(\cdot)$ indicates the minimum positive eigenvalue for matrices having positive eigenvalues. For any square matrix $A$, denote $\|x\|_A^2$=$x^\top Ax$.

\section{Preliminaries}\label{nonconvex:sec-preliminary}
In this section, we briefly introduce algebraic graph theory and the P--{\L} condition.

\subsection{Algebraic Graph Theory}

Let $\mathcal G=(\mathcal V,\mathcal E, A)$ denote a weighted undirected graph with the set of vertices (nodes) $\mathcal V =[n]$, the set of links (edges) $\mathcal E
\subseteq \mathcal V \times \mathcal V$, and the weighted adjacency matrix
$A =A^{\top}=(a_{ij})$ with nonnegative elements $a_{ij}$. A link of $\mathcal G$ is denoted by $(i,j)\in \mathcal E$ if $a_{ij}>0$, i.e., if vertices $i$ and $j$ can communicate with each other. It is assumed that $a_{ii}=0$ for all $i\in [n]$. Let $\mathcal{N}_i=\{j\in [n]:~ a_{ij}>0\}$ and $\deg_i=\sum\limits_{j=1}^{n}a_{ij}$ denote the neighbor set and weighted degree of vertex $i$, respectively. The degree matrix of graph $\mathcal G$ is $\Deg=\diag([\deg_1, \cdots, \deg_n])$. The Laplacian matrix is $L=(L_{ij})=\Deg-A$. A  path of length $k$ between vertices $i$ and $j$ is a subgraph with distinct vertices $i_0=i,\dots,i_k=j\in [n]$ and edges $(i_j,i_{j+1})\in\mathcal E,~j=0,\dots,k-1$.
An undirected graph is  connected if there exists at least one path between any two distinct vertices.

\subsection{Polyak--{\L}ojasiewicz Condition}
Let $f(x):~\mathbb{R}^d\mapsto\mathbb{R}$ be a differentiable function. Let $\mathbb{X}^*=\argmin_{x\in\mathbb{R}^p}f(x)$ and $f^*=\min_{x\in\mathbb{R}^d}f(x)$. Moreover, we assume that $f^*>-\infty$.
\begin{definition} 
	The function $f$ satisfies the Polyak--{\L}ojasiewicz (P--{\L}) condition with constant  $\nu>0$ if
	\begin{align}
		\frac{1}{2}\|\nabla f(x)\|^2\ge \nu( f(x)-f^*),~\forall x\in \mathbb{R}^d.\label{nonconvex:equ:plc}
	\end{align}
\end{definition}
It is straightforward to see that every (essentially or weakly) strongly convex function satisfies the P--{\L} condition.
The P--{\L} condition implies that every stationary point is a global minimizer. But unlike the (essentially or weakly) strong convexity, the P--{\L} condition alone does not imply convexity of $f$. Moreover, it does not imply that the global minimizer is unique either. In fact, P--{\L} condition generalizes strong convexity to nonconvex functions. The function $f(x)=x^2+3\sin^2(x)$ given in \cite{karimi2016linear} is an example of a nonconvex function satisfying the P--{\L} condition with $\nu=1/32$. Moreover, it was shown in \cite{li2018simple_nips} that the loss functions in some applications satisfy the P--{\L} condition in the local region near a local minimum. Moreover, \cite{fazel2018global} proved that the cost function of the policy optimization for the linear quadratic regulator problem is nonconvex and satisfies  the P--{\L} condition. 

\section{Problem Formulation}\label{nonconvex:sec-proformu}
In this section, we introduce three general classes of compressors and provide the assumptions on the communication network and cost functions.

\subsection{Compressors}
To improve communication efficiency, we consider the scenario that the communication between agents is compressed. Specifically, we consider a class of compressors with bounded relative compression error, and two classes of compressors respectively with globally and locally bounded absolute compression error satisfying the following assumptions.
\begin{assumption}\label{nonconvex:ass:compression}
	The compressor $\mathcal{C}:\mathbb{R}^d\mapsto\mathbb{R}^d$ satisfies
	\begin{align}\label{nonconvex:ass:compression_equ_scaling}
		\mathbf{E}_{\mathcal{C}}\Big[\Big\|\frac{\mathcal{C}(x)}{r}-x\Big\|^2\Big]\le (1-\varphi)\|x\|^2,~\forall x\in\mathbb{R}^d,
	\end{align}
	for some constants $\varphi\in(0,1]$ and $r>0$. Here $\mathbf{E}_{\mathcal{C}}[\cdot]$ denotes the expectation over the internal randomness of the stochastic compression operator $\mathcal{C}$.
\end{assumption}

From \eqref{nonconvex:ass:compression_equ_scaling}, we have
\begin{align}\label{nonconvex:ass:compression_equ}
	\mathbf{E}_{\mathcal{C}}[\|\mathcal{C}(x)-x\|^2]
	&=\mathbf{E}_{\mathcal{C}}\Big[\Big\|r\Big(\frac{\mathcal{C}(x)}{r}-x\Big)+(r-1)x\Big\|^2\Big]\nonumber\\
	&\le2r^2\mathbf{E}_{\mathcal{C}}\Big[\Big\|\frac{\mathcal{C}(x)}{r}-x\Big\|^2\Big]
	+2(1-r)^2\|x\|^2\nonumber\\
	&\le r_0\|x\|^2,~\forall x\in\mathbb{R}^d,
\end{align}
where $r_0=2r^2(1-\varphi)+2(1-r)^2$.
Therefore, the class of compressors satisfying Assumption~\ref{nonconvex:ass:compression} is the same as that used in \cite{li2021compressed}.
As explained in \cite{li2021compressed}, the class of compressors satisfying Assumption~\ref{nonconvex:ass:compression} is broad. It incorporates all the  unbiased  compressors with bounded relative compression error\footnote{A compressor $\mathcal{C}:\mathbb{R}^d\mapsto\mathbb{R}^d$ is  unbiased with bounded relative compression error (or just unbiased for simplicity) if for all $x\in\mathbb{R}^d$, $\mathbf{E}_{\mathcal{C}}[\mathcal{C}(x)]=x$ and there exists a constant $C\ge0$ such that $\mathbf{E}_{\mathcal{C}}[\|\mathcal{C}(x)-x\|^2]\le C\|x\|^2$.} and biased but contractive compressors\footnote{A compressor $\mathcal{C}:\mathbb{R}^d\mapsto\mathbb{R}^d$ is contractive if there exists a constant $\varphi\in(0,1]$ such that $\mathbf{E}_{\mathcal{C}}[\|\mathcal{C}(x)-x\|^2]\le (1-\varphi)\|x\|^2,~\forall x\in\mathbb{R}^d$.}, such as random quantization and sparsification, which are commonly used in the literature, e.g., \cite{alistarh2017qsgd,tang2018communication,horvath2019stochastic,reisizadeh2019exact,
	Koloskova2019decentralized,singh2020sparq,koloskova2020decentralized,taheri2020quantized,
	liu2021linear,kovalev2021linearly,singh2021squarm,chen2021communication}. It also includes some biased and non-contractive compressors, such as the norm-sign compressor. Moreover, it is straightforward to check that the class of compressors satisfying Assumption~\ref{nonconvex:ass:compression} also covers the three classes of biased compressors considered in \cite{beznosikov2020biased}. In other words, Assumption~\ref{nonconvex:ass:compression} is weaker than various commonly used assumptions for compressors in the literature.

\begin{assumption}\label{nonconvex:ass:compression_quantization}
	The compressor $\mathcal{C}:\mathbb{R}^d\mapsto\mathbb{R}^d$ satisfies
	\begin{align}\label{nonconvex:ass:compression_equ_quantization}
		\mathbf{E}_{\mathcal{C}}[\|\mathcal{C}(x)-x\|_p^2]\le C,~\forall x\in\mathbb{R}^d,
	\end{align}
	for some real number $p\ge1$ and constant $C\ge0$.
\end{assumption}
The same class of compressors satisfying Assumption~\ref{nonconvex:ass:compression_quantization} has also been used in \cite{khirirat2020compressed}, which incorporates the deterministic quantization used in \cite{nedic2008distributed,zhu2016quantized,yuan2012distributed} and the unbiased random quantization used in \cite{yuan2012distributed,tang2018communication,reisizadeh2019robust}. 

\begin{assumption}\label{nonconvex:ass:compression_determin}
	The compressor $\mathcal{C}:\mathbb{R}^d\mapsto\mathbb{R}^d$ satisfies
	\begin{align}\label{nonconvex:ass:compression_equ_determin}
		\|\mathcal{C}(x)-x\|_p\le (1-\varphi),~\forall x\in\{x\in\mathbb{R}^d:~\|x\|_p\le1\},
	\end{align}
	for some real number $p\ge1$ and constant $\varphi\in(0,1]$.
\end{assumption}
The same class of compressors satisfying Assumption~\ref{nonconvex:ass:compression_determin} has also been used in \cite{zhang2021innovation}, which covers the standard uniform quantizer with dynamic and fixed quantization levels respectively used in \cite{yi2014quantized,lee2018finite,magnusson2020maintaining,kajiyama2020linear,xiong2021quantized} and \cite{lei2020distributed}, and the Moniqua used in \cite{lu2020moniqua}. Moreover, as pointed out in \cite{zhang2021innovation}, the $1$-bit binary quantizer satisfies Assumption~\ref{nonconvex:ass:compression_determin}.
The difference between Assumptions~\ref{nonconvex:ass:compression_quantization} and \ref{nonconvex:ass:compression_determin} is that the former is a global assumption while the latter is a local assumption. It should be pointed out that all Assumptions~\ref{nonconvex:ass:compression}--\ref{nonconvex:ass:compression_determin} do not require the compressors to be unbiased.
Note that the inequalities in Assumptions~\ref{nonconvex:ass:compression}--\ref{nonconvex:ass:compression_determin} are different, and no one can imply another. Therefore, the three types of compressors are different from each other, and no one type is more restrictive than or can imply another. Moreover, the intersection of each pair of the three types of compressors is non-empty. For example, as explained in the Simulations, the norm-sign compressor satisfies both Assumptions~1 and 3. Therefore, the three types of compressors are not mutually exclusive.

The above three general classes of compressors cover most of existing compressors used in machine learning and signal processing applications, which substantiate the
generality of our results later in this paper.

\subsection{Communication Network and Cost Functions}
The following assumptions for the problem \eqref{nonconvex:eqn:xopt} are made.

\begin{assumption}\label{nonconvex:ass:graph}
	The underlying communication network is modeled by an undirected and connected graph $\mathcal G$.
\end{assumption}

\begin{assumption}\label{nonconvex:ass:optset}
	The minimum function value of the optimization problem \eqref{nonconvex:eqn:xopt} is finite.
\end{assumption}

\begin{assumption}\label{nonconvex:ass:fiu}
	Each  local cost function $f_i(x)$ is smooth with constant $L_{f}>0$, i.e., it is differentiable and
	\begin{align}\label{nonconvex:smooth}
		\|\nabla f_i(x)-\nabla f_i(y)\|\le L_{f}\|x-y\|,~\forall x,y\in \mathbb{R}^d.
	\end{align}
\end{assumption}

\begin{assumption}\label{nonconvex:ass:fil} The global cost function $f(x)$ satisfies the P--{\L} condition with constant $\nu>0$.
\end{assumption}

Assumptions~\ref{nonconvex:ass:graph}--\ref{nonconvex:ass:fiu} are standard in the literature to guarantee the well-known $\mathcal{O}(1/T)$ convergence rate for distributed algorithms finding the first-order stationary points for nonconvex optimization problems.
Assumption~\ref{nonconvex:ass:fil}  is  weaker than the assumption that the global or each local cost function is strongly convex, but it still can guarantee linear convergence. Note that the convexity of the cost functions and the boundedness of their gradients are not assumed. We also make no assumptions on the boundedness of the deviation between the gradients of local cost functions. In other words, we do not assume that $\frac{1}{n}\sum_{i=1}^{n}\|\nabla f_i(x)-\nabla f(x)\|^2$ is bounded. Moreover, we do not assume that the optimal set is a singleton or finite set either.


\section{Compressed Communication Algorithm: Bounded Relative Compression Error}\label{nonconvex:sec-main-dc}
In this section, we use the compressors with bounded relative compression error to design a communication-efficient distributed algorithm and analyze the convergence properties of the proposed algorithm.

\subsection{Algorithm Description}
To solve \eqref{nonconvex:eqn:xopt}, various distributed algorithms have been proposed. For example, \cite{yi2021linear} proposed the following distributed primal--dual algorithm:
\begin{subequations}\label{nonconvex:yi-tac-alg}
	\begin{align}
		x_{i,k+1} &= x_{i,k}-\eta\Big(\alpha\sum_{j=1}^{n}L_{ij}x_{j,k}+\beta v_{i,k}+\nabla f_i(x_{i,k})\Big), \label{nonconvex:yi-tac-alg-x}\\
		v_{i,k+1} &=v_{i,k}+ \eta\beta\sum_{j=1}^{n}L_{ij}x_{j,k},  \label{nonconvex:yi-tac-alg-v}
	\end{align}
\end{subequations}
where $\alpha$, $\beta$, and $\eta$ are positive algorithm parameters, and $x_{i,k}\in\mathbb{R}^d$ is agent $i$'s estimation of the solution to the problem \eqref{nonconvex:eqn:xopt} at the $k$-th iteration. 

To implement the algorithm \eqref{nonconvex:yi-tac-alg}, at each iteration each agent $j$ needs to exactly communicate the vector-valued variable $x_{j,k}$ with its neighbors, which requires significant amount of data exchange especially when the dimension $d$ is large. However, in practice communication bandwidth and power are limited, which motivates this paper to consider communication-efficient distributed algorithms. We use communication compression to improve communication efficiency. As mentioned in the Introduction, directly combining the algorithm \eqref{nonconvex:yi-tac-alg} and communication compression, i.e., using the compressed variable $\mathcal{C}(x_{j,k})$ to replace $x_{j,k}$ in \eqref{nonconvex:yi-tac-alg}, does not work due to the compression error. To reduce the compression error, an auxiliary variable $a_{j,k}\in\mathbb{R}^d$ is introduced. The difference $x_{j,k}-a_{j,k}$ instead of $x_{j,k}$ is compressed and communicated, and then is added back to $a_{j,k}$ for replacing $x_{j,k}$ in \eqref{nonconvex:yi-tac-alg}. Then, we have the following algorithm
\begin{subequations}\label{nonconvex:kia-algo-dc-compress}
	\begin{align}
		x_{i,k+1} &= x_{i,k}-\eta\Big(\alpha\sum_{j=1}^{n}L_{ij}\hat{x}_{j,k}+\beta v_{i,k}+\nabla f_i(x_{i,k})\Big), \label{nonconvex:kia-algo-dc-x-compress}\\
		v_{i,k+1} &=v_{i,k}+ \eta\beta\sum_{j=1}^{n}L_{ij}\hat{x}_{j,k},  \label{nonconvex:kia-algo-dc-v-compress}
	\end{align}
\end{subequations}
where 
\begin{align}
	\hat{x}_{i,k}=a_{i,k}+\mathcal{C}(x_{i,k}-a_{i,k}).\label{nonconvex:kia-algo-dc-compact-xhat}
\end{align}

\begin{algorithm}[!tb]
	\caption{}
	\label{nonconvex:algorithm-pdgd}
	\begin{algorithmic}[1]
		\STATE \textbf{Input}: positive parameters $\alpha$, $\beta$, $\eta$, and $\psi$.
		\STATE \textbf{Initialize}: $ x_{i,0}\in\mathbb{R}^d$, $a_{i,0}=b_{i,0}=v_{i,0}={\bf 0}_d$, and $q_{i,0}=\mathcal{C}(x_{i,0}),~\forall i\in[n]$.
		\FOR{$k=0,1,\dots$}
		\FOR{$i=1,\dots,n$  in parallel}
		\STATE  Broadcast $q_{i,k}$ to $\mathcal{N}_i$ and receive $q_{j,k}$ from $j\in\mathcal{N}_i$.
		\STATE  Update
		\begin{subequations}\label{nonconvex:kia-algo-dc}
			\begin{align}
				a_{i,k+1}&=a_{i,k}+\psi q_{i,k}, \label{nonconvex:kia-algo-dc-a}\\
				b_{i,k+1}&=b_{i,k}+\psi \Big(q_{i,k}-\sum_{j=1}^{n}L_{ij}q_{j,k}\Big), \label{nonconvex:kia-algo-dc-b}\\
				x_{i,k+1} &= x_{i,k}-\eta\alpha\Big(a_{i,k}-b_{i,k}+\sum_{j=1}^{n}L_{ij}q_{j,k}\Big)-\eta(\beta v_{i,k}+\nabla f_i(x_{i,k})), \label{nonconvex:kia-algo-dc-x}\\
				v_{i,k+1} &=v_{i,k}+ \eta\beta\Big(a_{i,k}-b_{i,k}+\sum_{j=1}^{n}L_{ij}q_{j,k}\Big),  \label{nonconvex:kia-algo-dc-v}\\
				q_{i,k+1}&=\mathcal{C}(x_{i,k+1}-a_{i,k+1}). \label{nonconvex:kia-algo-dc-q}
			\end{align}
		\end{subequations}
		\ENDFOR
		\ENDFOR
		\STATE  \textbf{Output}: $\{x_{i,k}\}$.
	\end{algorithmic}
\end{algorithm}

Although in the algorithm \eqref{nonconvex:kia-algo-dc-compress}, the compressor error can be reduced, at each iteration each agent $j$ still needs to exactly communicate the vector-valued variable $a_{j,k}$ due to the summation term $\sum_{j=1}^{n}L_{ij}a_{j,k}$ inside \eqref{nonconvex:kia-algo-dc-compress}. Thus, the algorithm \eqref{nonconvex:kia-algo-dc-compress} does not enjoy the benefits of compression. To overcome that, another auxiliary variable $b_{j,k}\in\mathbb{R}^d$ is introduced to calculate $\sum_{j=1}^{n}L_{ij}a_{j,k}$. The proposed algorithm is presented in pseudo-code as Algorithm~\ref{nonconvex:algorithm-pdgd}, which is a communication-efficient algorithm since each agent $j$ only communicates the compressed variable $q_{j,k}$ with its neighbors.
Noting that $a_{i,0}=b_{i,0}={\bf 0}_d$, by mathematical induction, it is straightforward to check that $b_{i,k}=a_{i,k}-\sum_{j=1}^{n}L_{ij}a_{j,k},~\forall i\in[n]$. Then, \eqref{nonconvex:kia-algo-dc-x} and \eqref{nonconvex:kia-algo-dc-v} respectively can be rewritten as \eqref{nonconvex:kia-algo-dc-x-compress} and \eqref{nonconvex:kia-algo-dc-v-compress}. The same idea to use auxiliary variables to reduce the compression error and to implement communication compression has been used in the literature, e.g., \cite{liu2021linear,li2021compressed}.

To end this section, we would like to briefly explain why the compression error is reduced in Algorithm~\ref{nonconvex:algorithm-pdgd} when the class of compressors satisfying Assumption~\ref{nonconvex:ass:compression} is used. 	From \eqref{nonconvex:kia-algo-dc-compact-xhat} and \eqref{nonconvex:ass:compression_equ}, we have
\begin{align}\label{nonconvex:xminusxhat}
	\mathbf{E}_{\mathcal{C}}[\|x_{i,k}-\hat{x}_{i,k}\|^2]
	&=\mathbf{E}_{\mathcal{C}}[\|x_{i,k}-a_{i,k}-\mathcal{C}(x_{i,k}-a_{i,k})\|^2]\le r_0\mathbf{E}_{\mathcal{C}}[\|x_{i,k}-a_{i,k}\|^2].
\end{align}
From the proof of Theorem~\ref{nonconvex:thm-sm}, we know that $\sum_{i=1}^{n}\mathbf{E}_{\mathcal{C}}[\|x_{i,k}-a_{i,k}\|^2]$ converges to zero. Therefore, the compression error is reduced.


\subsection{Convergence Analysis}
In this section, we provide convergence analysis for both scenarios without and with Assumption~\ref{nonconvex:ass:fil}.
We first have the following convergence result.
\begin{theorem}\label{nonconvex:thm-sm}
	Suppose that Assumptions~\ref{nonconvex:ass:compression} and \ref{nonconvex:ass:graph}--\ref{nonconvex:ass:fiu} hold.  Let $\{x_{i,k}\}$ be the sequence generated by Algorithm~\ref{nonconvex:algorithm-pdgd} with $\alpha=\kappa_1\beta$, $\beta>\kappa_2$, $\eta\in(0,\kappa_3)$, and $\psi\in(0,1/r]$, where $\kappa_1,~\kappa_2,~\kappa_3$ are positive constants given in Appendix~\ref{nonconvex:proof-thm-sm}. Then, for any $T\in\mathbb{N}_0$,
	\begin{subequations}
		\begin{align}
			&\sum_{k=0}^{T}\sum_{i=1}^{n}\mathbf{E}_{\mathcal{C}}[\|x_{i,k}-\bar{x}_k\|^2+\|\nabla f(\bar{x}_k)\|^2]=\mathcal{O}(1),\label{nonconvex:thm-sm-equ1}\\
			&\mathbf{E}_{\mathcal{C}}[f(\bar{x}_{T})-f^*]=\mathcal{O}(1),\label{nonconvex:thm-sm-equ2}
		\end{align}
	\end{subequations}
	where $\bar{x}_k=\frac{1}{n}\sum_{i=1}^{n}x_{i,k}$.
\end{theorem}
\begin{proof}
	We use Lyapunov analysis to prove this theorem. More specifically, we first appropriately design a function $V_k$ which contains terms $\sum_{i=1}^{n}\mathbf{E}_{\mathcal{C}}[\|x_{i,k}-\bar{x}_k\|^2]$, $\mathbf{E}_{\mathcal{C}}[f(\bar{x}_{T})-f^*]$, and $\sum_{i=1}^{n}\mathbf{E}_{\mathcal{C}}[\|x_{i,k}-a_{i,k}\|^2]$ describing consensus, optimization, and compression errors, respectively.  We then prove that $\mathbf{E}_{\mathcal{C}}[V_k]$ is non-increasing by showing that the difference $\mathbf{E}_{\mathcal{C}}[V_k-V_{k+1}]$ can be lower bounded by $\sum_{i=1}^{n}\mathbf{E}_{\mathcal{C}}[\|x_{i,k}-\bar{x}_k\|^2+\|x_{i,k}-a_{i,k}\|^2+\|\nabla f(\bar{x}_k)\|^2]$. 
	We finally show that $V_k$ is non-negative and get \eqref{nonconvex:thm-sm-equ1}--\eqref{nonconvex:thm-sm-equ2} by summarizing the inequalities containing the difference $\mathbf{E}_{\mathcal{C}}[V_k-V_{k+1}]$. 
	The explicit expressions of the Lyapunov function $V_k$ and the right-hand sides of \eqref{nonconvex:thm-sm-equ1}--\eqref{nonconvex:thm-sm-equ2}, and the  detailed proof are given in Appendix~\ref{nonconvex:proof-thm-sm}. 
\end{proof}

We have several remarks on Theorem~\ref{nonconvex:thm-sm}.
Firstly, from \eqref{nonconvex:thm-sm-equ1}, we know that $\min_{k\in[T]}\{\sum_{i=1}^{n}\mathbf{E}_{\mathcal{C}}[\|x_{i,k}-\bar{x}_k\|^2+\|\nabla f(\bar{x}_k)\|^2]\}=\mathcal{O}(1/T)$. In other words, Algorithm~\ref{nonconvex:algorithm-pdgd} finds a first-order stationary point with the well-known rate $\mathcal{O}(1/T)$, which is the same as that achieved by the distributed algorithms with exact communication in the literature, e.g, \cite{yi2021linear,daneshmand2018second,hong2017prox,sun2019distributed,hajinezhad2019perturbed}. Secondly, from \eqref{nonconvex:thm-sm-equ2}, we know that the cost difference between the global optimum and the resulting stationary point is bounded. Thirdly, it should be pointed out that the settings on the parameters  $\alpha$, $\beta$, and $\eta$ are just sufficient conditions. With some modifications of the proofs, other forms of settings for these algorithm parameters still can guarantee the same type of convergence result. Fourthly, observe that the definitions of $\kappa_1$ and $\kappa_2$ given in Appendix~\ref{nonconvex:proof-thm-sm} are independent of the parameters related to the compressors. Therefore, the choice of the parameters $\alpha$ and $\beta$ is independent of the compressors.  Finally, the proof of Theorem~\ref{nonconvex:thm-sm} is inspired by the proof of Theorem~1 in \cite{yi2021linear}. However, due to the compressed compression, a different Lyapunov function is appropriately designed and the details are also different.

Then, with Assumption~\ref{nonconvex:ass:fil}, the following result states that Algorithm~\ref{nonconvex:algorithm-pdgd} can linearly find a global optimum.
\begin{theorem}\label{nonconvex:thm-ft}
	Suppose that Assumptions~\ref{nonconvex:ass:compression} and \ref{nonconvex:ass:graph}--\ref{nonconvex:ass:fil} hold. Let $\{x_{i,k}\}$ be the sequence generated by Algorithm~\ref{nonconvex:algorithm-pdgd} with  the same $\alpha$, $\beta$, $\eta$, and $\psi$ given in Theorem~\ref{nonconvex:thm-sm}. Then, for any $k\in\mathbb{N}_0$,
	\begin{align}\label{nonconvex:thm-ft-equ1}
		\sum_{i=1}^{n}\mathbf{E}_{\mathcal{C}}[\|x_{i,k}-\bar{x}_k\|^2+f(\bar{x}_k)-f^*]
		=\mathcal{O}((1-\epsilon)^{k}),
	\end{align}
	where $\epsilon$ is a constant in $(0,1)$ given in Appendix~\ref{nonconvex:proof-thm-ft}.
\end{theorem}
\begin{proof}
	This proof is based on the proof of Theorem~\ref{nonconvex:thm-sm}. From the P--{\L} condition, we know that $\mathbf{E}_{\mathcal{C}}[\|\nabla f(\bar{x}_k)\|^2]$ can be lower bounded by $\mathbf{E}_{\mathcal{C}}[f(\bar{x}_k)-f^*]$, which further implies the difference $\mathbf{E}_{\mathcal{C}}[V_k-V_{k+1}]$ can be lower bounded by $\mathbf{E}_{\mathcal{C}}[V_k]$. Therefore, $\mathbf{E}_{\mathcal{C}}[V_k]$ exponentially decreases to zero. Thus, \eqref{nonconvex:thm-ft-equ1} holds.
	The explicit expression of the right-hand side of \eqref{nonconvex:thm-ft-equ1} and the detailed proof are given in Appendix~\ref{nonconvex:proof-thm-ft}.
\end{proof}

We have several remarks on Theorem~\ref{nonconvex:thm-ft}.
Firstly, observe that Algorithm~\ref{nonconvex:algorithm-pdgd} uses the same algorithm parameters for the cases without and with the P--{\L} condition in Theorems~\ref{nonconvex:thm-sm} and \ref{nonconvex:thm-ft}, respectively. As a result, it is not needed to check the P--{\L} condition before  implementing Algorithm~\ref{nonconvex:algorithm-pdgd}, which is important since it is normally difficult to check that condition. Secondly, compared to \cite{li2021compressed} which used the same type of compressors and established linear convergence under the condition that the global cost function is strongly convex, we show linear convergence under the weaker P--{\L} condition and only use a half number of compression and communication operations per iteration since in the algorithm proposed in \cite{li2021compressed} each agent needs to communicate two compressed variables with its neighbors. Thirdly, compared to \cite{liu2021linear,kovalev2021linearly} which used unbiased compressors with  bounded
relative compression error and established linear convergence under the condition that each local cost function is strongly convex, we use the more general compressors and the weaker P--{\L} condition to show linear convergence. Lastly, compared to \cite{reisizadeh2019exact} which used unbiased compressors with  bounded
relative compression error but only achieved sublinear convergence under the condition that each local cost function is strongly convex, we not only use the more general compressors and the weaker P--{\L} condition, but also show strictly faster convergence.

\section{Error Feedback Based Compressed Communication Algorithm: Bounded Relative Compression Error}\label{nonconvex:sec-main-dc_ef}
In this section, we extend Algorithm~\ref{nonconvex:algorithm-pdgd} to error feedback version for biased compressors particularly.

\subsection{Algorithm Description}
\begin{algorithm}[tb]
	\caption{}
	\label{nonconvex:algorithm-pdgd_ef}
	\begin{algorithmic}[1]
		\STATE \textbf{Input}: positive parameters $\alpha$, $\beta$, $\eta$, $\psi$, and $\sigma$.
		\STATE \textbf{Initialize}: $ x_{i,0}\in\mathbb{R}^d$, $a_{i,0}=b_{i,0}=e_{i,0}=v_{i,0}={\bf 0}_d$, and $q_{i,0}=\hat{q}_{i,0}=\mathcal{C}(x_{i,0}),~\forall i\in[n]$.
		\FOR{$k=0,1,\dots$}
		\FOR{$i=1,\dots,n$  in parallel}
		\STATE  Broadcast $q_{i,k}$ and $\hat{q}_{i,k}$ to $\mathcal{N}_i$ and receive $q_{j,k}$ and $\hat{q}_{j,k}$ from $j\in\mathcal{N}_i$.
		\STATE  Update
		\begin{subequations}\label{nonconvex:kia-algo-dc_ef}
			\begin{align}
				a_{i,k+1}&=a_{i,k}+\psi q_{i,k}, \label{nonconvex:kia-algo-dc_ef-a}\\
				b_{i,k+1}&=b_{i,k}+\psi \Big(q_{i,k}-\sum_{j=1}^{n}L_{ij}q_{j,k}\Big), \label{nonconvex:kia-algo-dc_ef-b}\\
				x_{i,k+1} &= x_{i,k}-\eta\alpha\Big(a_{i,k}-b_{i,k}+\sum_{j=1}^{n}L_{ij}\hat{q}_{j,k}\Big)-\eta(\beta v_{i,k}+\nabla f_i(x_{i,k})), \label{nonconvex:kia-algo-dc_ef-x}\\
				v_{i,k+1} &=v_{i,k}+ \eta\beta\Big(a_{i,k}-b_{i,k}+\sum_{j=1}^{n}L_{ij}\hat{q}_{j,k}\Big),  \label{nonconvex:kia-algo-dc_ef-v}\\
				q_{i,k+1}&=\mathcal{C}(x_{i,k+1}-a_{i,k+1}), \label{nonconvex:kia-algo-dc_ef-q}\\
				e_{i,k+1}&=\sigma e_{i,k}+x_{i,k}-a_{i,k}-\hat{q}_{i,k}, \label{nonconvex:kia-algo-dc_ef-e}\\
				\hat{q}_{i,k+1}&=\mathcal{C}(\sigma e_{i,k+1}+x_{i,k+1}-a_{i,k+1}). \label{nonconvex:kia-algo-dc_ef-qhat}
			\end{align}
		\end{subequations}
		\ENDFOR
		\ENDFOR
		\STATE  \textbf{Output}: $\{x_{i,k}\}$.
	\end{algorithmic}
\end{algorithm}

The error feedback based communication-efficient distributed algorithm is presented in pseudo-code as Algorithm~\ref{nonconvex:algorithm-pdgd_ef}.
Without ambiguity, we denote
\begin{align}
	\hat{x}_{i,k}=a_{i,k}+\hat{q}_{i,k},\label{nonconvex:kia-algo-dc_ef-compact-xhat}
\end{align}
then \eqref{nonconvex:kia-algo-dc_ef-x} and \eqref{nonconvex:kia-algo-dc_ef-v} respectively can be written as \eqref{nonconvex:kia-algo-dc-x-compress} and \eqref{nonconvex:kia-algo-dc-v-compress} since $b_{i,k}=a_{i,k}-\sum_{j=1}^{n}L_{ij}a_{j,k},~\forall i\in[n]$. Therefore, Algorithm~\ref{nonconvex:algorithm-pdgd_ef} also is a communication-efficient extension of the distributed primal--dual algorithm  \eqref{nonconvex:yi-tac-alg}.

Compared to Algorithm~\ref{nonconvex:algorithm-pdgd},  Algorithm~\ref{nonconvex:algorithm-pdgd_ef} has two new variables $\hat{q}_{i,k}$ and $e_{i,k}$ which are used to estimate the biased compression error and accumulate the biased compression errors, respectively. Then each agent can use $\hat{q}_{i,k}$ to correct the bias induced by the biased compressors\footnote{For unbiased compressors, it is unnecessary to consider error feedback since $\mathbf{E}_{\mathcal{C}}[e_{i,k}]={\bf 0}_d$.}. However, compared to Algorithm~\ref{nonconvex:algorithm-pdgd}, there are twice number of compression and communication operations per iteration in Algorithm~\ref{nonconvex:algorithm-pdgd_ef}.

\subsection{Convergence Analysis}
Similar to Theorem~\ref{nonconvex:thm-sm}, we first have the following sublinear convergence result for Algorithm~\ref{nonconvex:algorithm-pdgd_ef} without Assumption~\ref{nonconvex:ass:fil}.
\begin{theorem}\label{nonconvex:thm-sm_ef}
	Suppose that Assumptions~\ref{nonconvex:ass:compression} and \ref{nonconvex:ass:graph}--\ref{nonconvex:ass:fiu} hold.  Let $\{x_{i,k}\}$ be the sequence generated by Algorithm~\ref{nonconvex:algorithm-pdgd_ef} with $\alpha=\kappa_1\beta$, $\beta>\kappa_2$, $\eta\in(0,\check{\kappa}_3)$, $\sigma\in(0,\kappa_0)$ and $\psi\in(0,1/r]$, where $\kappa_1,~\kappa_2$ and $\kappa_0,~\check{\kappa}_3$ are positive constants given in Appendices~\ref{nonconvex:proof-thm-sm} and \ref{nonconvex:proof-thm-sm_ef}, respectively. Then, for any $T\in\mathbb{N}_0$,
	\begin{subequations}
		\begin{align}
			&\sum_{k=0}^{T}\sum_{i=1}^{n}\mathbf{E}_{\mathcal{C}}[\|x_{i,k}-\bar{x}_k\|^2+\|\nabla f(\bar{x}_k)\|^2]=\mathcal{O}(1),\label{nonconvex:thm-sm_ef-equ1}\\
			&\mathbf{E}_{\mathcal{C}}[f(\bar{x}_{T})-f^*]=\mathcal{O}(1).\label{nonconvex:thm-sm_ef-equ2}
		\end{align}
	\end{subequations}
\end{theorem}
\begin{proof}
	This proof is similar to the proof of Theorem~\ref{nonconvex:thm-sm}, but uses a different Lyapunov function $W_k$.
	Due to space limitations, the explicit expressions of the Lyapunov function $W_k$ and the right-hand sides of \eqref{nonconvex:thm-sm_ef-equ1}--\eqref{nonconvex:thm-sm_ef-equ2}, and the detailed proof are given in Appendix~\ref{nonconvex:proof-thm-sm_ef}.
\end{proof}

Similar to Theorem~\ref{nonconvex:thm-ft}, we then have the following linear convergence result for Algorithm~\ref{nonconvex:algorithm-pdgd_ef} with Assumption~\ref{nonconvex:ass:fil}.
\begin{theorem}\label{nonconvex:thm-ft_ef}
	Suppose that Assumptions~\ref{nonconvex:ass:compression} and \ref{nonconvex:ass:graph}--\ref{nonconvex:ass:fil} hold. Let $\{x_{i,k}\}$ be the sequence generated by Algorithm~\ref{nonconvex:algorithm-pdgd_ef} with  the same $\alpha$, $\beta$, $\eta$, $\sigma$, and $\psi$ given in Theorem~\ref{nonconvex:thm-sm_ef}. Then, for any $k\in\mathbb{N}_0$,
	\begin{align}\label{nonconvex:thm-ft_ef-equ1}
		\sum_{i=1}^{n}\mathbf{E}_{\mathcal{C}}[\|x_{i,k}-\bar{x}_k\|^2+f(\bar{x}_k)-f^*]
		=\mathcal{O}((1-\check{\epsilon})^{k}),
	\end{align}
	where $\check{\epsilon}$ is a constant in $(0,1)$ given in Appendix~\ref{nonconvex:proof-thm-ft_ef}.
\end{theorem}
\begin{proof}
	This proof is similar to the proof of Theorem~\ref{nonconvex:thm-ft}, but uses the Lyapunov function $W_k$ as used in the proof of Theorem~\ref{nonconvex:thm-sm_ef}.
	Due to space limitations, the expression of the right-hand side of \eqref{nonconvex:thm-ft_ef-equ1} and the detailed proof are given in  Appendix~\ref{nonconvex:proof-thm-ft_ef}.
\end{proof}

\section{Compressed Communication Algorithm: Bounded Absolute Compression Error}\label{nonconvex:sec-main-dc_determin}
In this section, we use the compressors with bounded absolute compression error to design a communication-efficient distributed algorithm and analyze the convergence properties of the proposed algorithm in various setups.

\subsection{Algorithm Description}
The communication-efficient distributed algorithm is presented in pseudo-code as Algorithm~\ref{nonconvex:algorithm-pdgd_determin}.
\begin{algorithm}[tb]
	\caption{}
	\label{nonconvex:algorithm-pdgd_determin}
	\begin{algorithmic}[1]
		\STATE \textbf{Input}: positive parameters $\alpha$, $\beta$, $\eta$, and a positive scaling sequence $\{s_k\}$.
		\STATE \textbf{Initialize}: $ x_{i,0}\in\mathbb{R}^d$, $\hat{x}_{i,-1}=y_{i,-1}=v_{i,0}={\bf 0}_d$, and $q_{i,0}=\mathcal{C}(x_{i,0}/s_0),~\forall i\in[n]$.
		\FOR{$k=0,1,\dots$}
		\FOR{$i=1,\dots,n$  in parallel}
		\STATE  Broadcast $q_{i,k}$ to $\mathcal{N}_i$ and receive $q_{j,k}$ from $j\in\mathcal{N}_i$.
		\STATE  Update
		\begin{subequations}\label{nonconvex:kia-algo-dc_determin}
			\begin{align}
				\hat{x}_{i,k}&=\hat{x}_{i,k-1}+s_kq_{i,k}, \label{nonconvex:kia-algo-dc-xhat_determin}\\
				y_{i,k}&=y_{i,k-1}+s_kq_{i,k}-s_k\sum_{j=1}^{n}L_{ij}q_{j,k}, \label{nonconvex:kia-algo-dc-y_determin}\\
				x_{i,k+1} &= x_{i,k}-\eta\alpha(\hat{x}_{i,k}-y_{i,k})-\eta(\beta v_{i,k}+\nabla f_i(x_{i,k})), \label{nonconvex:kia-algo-dc-x_determin}\\
				v_{i,k+1} &=v_{i,k}+ \eta\beta(\hat{x}_{i,k}-y_{i,k}),  \label{nonconvex:kia-algo-dc-v_determin}\\
				q_{i,k+1}&=\mathcal{C}((x_{i,k+1}-\hat{x}_{i,k})/s_{k+1}). \label{nonconvex:kia-algo-dc-q_determin}
			\end{align}
		\end{subequations}
		\ENDFOR
		\ENDFOR
		\STATE  \textbf{Output}: $\{x_{i,k}\}$.
	\end{algorithmic}
\end{algorithm}
By mathematical induction, it is straightforward to check that $y_{i,k}=\hat{x}_{i,k}-\sum_{j=1}^{n}L_{ij}\hat{x}_{j,k},~\forall i\in[n]$. Therefore, \eqref{nonconvex:kia-algo-dc-x_determin} and \eqref{nonconvex:kia-algo-dc-v_determin} can be rewritten as \eqref{nonconvex:kia-algo-dc-x-compress} and \eqref{nonconvex:kia-algo-dc-v-compress}, respectively. Therefore, Algorithm~\ref{nonconvex:algorithm-pdgd_determin} also is a communication-efficient extension of the distributed primal--dual algorithm \eqref{nonconvex:yi-tac-alg}. Moreover, same as Algorithm~\ref{nonconvex:algorithm-pdgd}, in Algorithm~\ref{nonconvex:algorithm-pdgd_determin} each agent only communicates one compressed variable with its neighbors. The difference between Algorithms~\ref{nonconvex:algorithm-pdgd} and \ref{nonconvex:algorithm-pdgd_determin} is that they use different types of compressors.

\subsection{Convergence Analysis}

We first analyze the performance of Algorithm~\ref{nonconvex:algorithm-pdgd_determin} when the class of compressors satisfying Assumption~\ref{nonconvex:ass:compression_quantization} is used.
Before stating the convergence results, we would like to briefly explain why this algorithm works.
From \eqref{nonconvex:kia-algo-dc-xhat_determin}, \eqref{nonconvex:kia-algo-dc-q_determin}, and \eqref{nonconvex:ass:compression_equ_quantization}, we have
\begin{align}\label{nonconvex:xminusxhat_quantization2}
	\mathbf{E}_{\mathcal{C}}[\|x_{i,k}-\hat{x}_{i,k}\|_p^2]
	&=\mathbf{E}_{\mathcal{C}}[\|x_{i,k}-\hat{x}_{i,k-1}
	-s_k\mathcal{C}((x_{i,k}-\hat{x}_{i,k-1})/s_k)\|_p^2]\nonumber\\
	&=\mathbf{E}_{\mathcal{C}}[s_k^2\|(x_{i,k}-\hat{x}_{i,k-1})/s_k
	-\mathcal{C}((x_{i,k}-\hat{x}_{i,k-1})/s_k)\|_p^2]\nonumber\\
	&\le Cs_k^2.
\end{align}
If we let $s_k$ exponentially decrease to zero, then the error caused by the compressed communication is neglectable. In this case, Algorithm~\ref{nonconvex:algorithm-pdgd_determin} using the second class of  compressors can have comparable convergence properties as the corresponding algorithm with exact communication, i.e., \eqref{nonconvex:yi-tac-alg}.

Similar to Theorem~\ref{nonconvex:thm-sm}, we have the following sublinear convergence result.

\begin{theorem}\label{nonconvex:thm-sm_quantization}
	Suppose that Assumptions~\ref{nonconvex:ass:compression_quantization} and \ref{nonconvex:ass:graph}--\ref{nonconvex:ass:fiu} hold.  Let $\{x_{i,k}\}$ be the sequence generated by Algorithm~\ref{nonconvex:algorithm-pdgd_determin} with $\alpha=\kappa_1\beta$, $\beta>\kappa_2$, $\eta\in(0,\tilde{\kappa}_3)$, and $s_k=s_0\gamma^k$, where $\kappa_1$ and $\kappa_2$ are positive constants given in Appendix~\ref{nonconvex:proof-thm-sm}, $\tilde{\kappa}_3$ a positive constant given in Appendix~\ref{nonconvex:proof-thm-sm_quantization}, $s_0$ is an arbitrary positive constant, and $\gamma$ is an arbitrary constant in $(0,1)$. Then, for any $T\in\mathbb{N}_0$,
	\begin{subequations}
		\begin{align}
			&\sum_{k=0}^{T}\sum_{i=1}^{n}\mathbf{E}_{\mathcal{C}}[\|x_{i,k}-\bar{x}_k\|^2+\|\nabla f(\bar{x}_k)\|^2]=\mathcal{O}(1),\label{nonconvex:thm-sm-equ1_quantization}\\
			&\mathbf{E}_{\mathcal{C}}[f(\bar{x}_{T})-f^*]=\mathcal{O}(1).\label{nonconvex:thm-sm-equ2_quantization}
		\end{align}
	\end{subequations}
\end{theorem}
\begin{proof}
	This proof is similar to the proof of Theorem~\ref{nonconvex:thm-sm}, but uses the non-negative function $U_k$ which contains terms describing consensus and optimization errors and is given in Appendix~\ref{nonconvex:proof-thm-sm}. We show that the difference $\mathbf{E}_{\mathcal{C}}[U_{k+1}-U_k]$ can be lower bounded by $\sum_{i=1}^{n}\mathbf{E}_{\mathcal{C}}[\|x_{i,k}-\bar{x}_k\|^2-\|x_{i,k}-\hat{x}_{i,k}\|^2+\|\nabla f(\bar{x}_k)\|^2]$. From \eqref{nonconvex:xminusxhat_quantization2} and $s_k=s_0\gamma^k$, we can get \eqref{nonconvex:thm-sm-equ1_quantization}--\eqref{nonconvex:thm-sm-equ2_quantization} by summarizing the inequalities containing the difference $\mathbf{E}_{\mathcal{C}}[U_{k+1}-U_k]$.
	Due to space limitations, the explicit expressions of the right-hand sides of \eqref{nonconvex:thm-sm-equ1_quantization}--\eqref{nonconvex:thm-sm-equ2_quantization} and the detailed proof are given Appendix~\ref{nonconvex:proof-thm-sm_quantization}.
\end{proof}

The remarks after Theorem~\ref{nonconvex:thm-sm} are still valid for Theorem~\ref{nonconvex:thm-sm_quantization}. Moreover, we would like to point out that the choice of the parameter $\gamma$ is also independent of the compressors since the definition of $\tilde{\kappa}_3$  given in Appendix~\ref{nonconvex:proof-thm-sm_quantization} is independent of the parameters related to the compressors. 

Similar to Theorem~\ref{nonconvex:thm-ft}, we then have the following linear convergence result for Algorithm~\ref{nonconvex:algorithm-pdgd_determin} when the class of compressors satisfying Assumption~\ref{nonconvex:ass:compression_quantization} is used.
\begin{theorem}\label{nonconvex:thm-ft_quantization}
	Suppose that Assumptions~\ref{nonconvex:ass:compression_quantization} and \ref{nonconvex:ass:graph}--\ref{nonconvex:ass:fil} hold. Let $\{x_{i,k}\}$ be the sequence generated by Algorithm~\ref{nonconvex:algorithm-pdgd_determin} with  the same $\alpha$, $\beta$, $\eta$, and $s_k$ given in Theorem~\ref{nonconvex:thm-sm_quantization}. Then, for any $k\in\mathbb{N}_0$,
	\begin{align}\label{nonconvex:thm-ft-equ1_quantization}
		\sum_{i=1}^{n}\mathbf{E}_{\mathcal{C}}[\|x_{i,k}-\bar{x}_k\|^2+f(\bar{x}_k)-f^*]
		=\mathcal{O}((1-\tilde{\epsilon})^{k}),
	\end{align}
	where $\tilde{\epsilon}$ is a constant in $(0,1)$ given in Appendix~\ref{nonconvex:proof-thm-ft_quantization}.
\end{theorem}
\begin{proof}
	This proof is based on the proof of Theorem~\ref{nonconvex:thm-sm_quantization}. From the P--{\L} condition, we know that $\mathbf{E}_{\mathcal{C}}[\|\nabla f(\bar{x}_k)\|^2]$ can be lower bounded by $\mathbf{E}_{\mathcal{C}}[f(\bar{x}_k)-f^*]$, which further implies the difference $\mathbf{E}_{\mathcal{C}}[U_{k+1}-U_k]$ can be lower bounded by $\mathbf{E}_{\mathcal{C}}[U_k-\sum_{i=1}^{n}\|x_{i,k}-\hat{x}_{i,k}\|^2]$. Then, combining this, \eqref{nonconvex:xminusxhat_quantization2}, and $s_k=s_0\gamma^k$, we can get that $\mathbf{E}_{\mathcal{C}}[U_k]$ exponentially decreases to zero. Thus, \eqref{nonconvex:thm-ft-equ1_quantization} holds.
	Due to space limitations, the explicit expression of the right-hand side of \eqref{nonconvex:thm-ft-equ1_quantization} and the detailed proof are given in Appendix~\ref{nonconvex:proof-thm-ft_quantization}.
\end{proof}

Compared to \cite{khirirat2020compressed} which used the same class of compressors satisfying Assumption~\ref{nonconvex:ass:compression_quantization}, Theorem~\ref{nonconvex:thm-ft_quantization} shows that a global optimum can be precisely found with a linear convergence rate under the P--{\L} condition. In contrast, although \cite{khirirat2020compressed} assumed the stronger strong convexity assumption and also showed that convergence rate is linear, the parallel algorithms proposed in \cite{khirirat2020compressed} only converged to a neighbor of the unique optimal point.

We also have the following linear convergence result for Algorithm~\ref{nonconvex:algorithm-pdgd_determin}  when the class of compressors satisfying Assumption~\ref{nonconvex:ass:compression_determin} is used.
\begin{theorem}\label{nonconvex:thm-ft_determin}
	Suppose that Assumptions~\ref{nonconvex:ass:compression_determin}--\ref{nonconvex:ass:fil} hold and a lower bound on  the P--{\L} constant $\nu$ is known in advance. Let $\{x_{i,k}\}$ be the sequence generated by Algorithm~\ref{nonconvex:algorithm-pdgd_determin} with $\alpha=\kappa_1\beta$, $\beta>\kappa_2$, $\eta\in(0,\hat{\kappa}_3)$, and $s_k=s_0\gamma^k$, where $\kappa_1$ and $\kappa_2$ are positive constants given in Appendix~\ref{nonconvex:proof-thm-sm}, $\hat{\kappa}_3$, $s_0$, and $\gamma$ are positive constants given in Appendix~\ref{nonconvex:proof-thm-ft_determin} with $\gamma\in(0,1)$. Then, for any $k\in\mathbb{N}_0$,
	\begin{align}\label{nonconvex:thm-ft-equ1_determin}
		\sum_{i=1}^{n}(\|x_{i,k}-\bar{x}_k\|^2+f(\bar{x}_k)-f^*)
		=\mathcal{O}(\gamma^{k}).
	\end{align}
\end{theorem}
\begin{proof}
	This proof also uses the non-negative function $U_k$ used in the proof of Theorem~\ref{nonconvex:thm-sm_quantization}. Note that the inequality \eqref{nonconvex:ass:compression_equ_determin} in Assumption~\ref{nonconvex:ass:compression_determin} only holds locally. We use  mathematical induction to prove that $\max_{i\in[n]}\|(x_{i,k}-\hat{x}_{i,k-1})/s_k\|_p^2\le 1$ and $U_k/s_k^2$ is globally bounded. Thus, \eqref{nonconvex:thm-ft-equ1_determin} holds.
	Due to space limitations, the explicit expression of the right-hand side of \eqref{nonconvex:thm-ft-equ1_determin} and the detailed proof are given in Appendix~\ref{nonconvex:proof-thm-ft_determin}.
\end{proof}

We have several remarks on Theorem~\ref{nonconvex:thm-ft_determin}. Firstly, compared to Theorems~\ref{nonconvex:thm-ft}, \ref{nonconvex:thm-ft_ef}, and \ref{nonconvex:thm-ft_quantization},  Theorem~\ref{nonconvex:thm-ft_determin} needs a lower bound on the P--{\L} constant $\nu$ to be known in advance, which is used to design the parameters $s_0$ and $\gamma$ as shown in Appendix~\ref{nonconvex:proof-thm-ft_determin}. This is a potential drawback since this constant is normally unknown due to the difficulty to check the P--{\L} condition. However, for strongly convex cost functions, this is not a drawback since if a function is strongly convex with convex parameter $\nu$, then it also satisfies the P--{\L} condition with the same constant $\nu$. Secondly, linear convergence has also been established in \cite{zhang2021innovation} which used the same type of compressors. However, \cite{zhang2021innovation} assumed that each local cost function is strongly convex, which is stronger than the condition that the global cost function satisfies the P--{\L} condition as used in Theorem~\ref{nonconvex:thm-ft_determin}, and required that the absolute compression error satisfies an inequality determined by the number of agents and the communication network, which is not needed in Theorem~\ref{nonconvex:thm-ft_determin}. Moreover, \cite{zhang2021innovation} required an unpractical condition that the unique optimal point needs to be known a priori to design algorithm parameters, which is a drawback. 
Thirdly, compared to \cite{lee2018finite,magnusson2020maintaining,kajiyama2020linear,xiong2021quantized} which used the standard uniform quantizer with dynamic quantization level and established linear convergence under the condition that each local cost function is strongly convex, we use the more general compressors and the weaker P--{\L} condition to show linear convergence. Finally, compared to \cite{lei2020distributed} which used the standard uniform quantizer with fixed quantization level and established linear convergence under the assumption that each local cost function is quadratic and the global cost function is strongly convex, we not only use the more general compressors but also consider the more general nonconvex functions satisfying the weaker P--{\L} condition.

To end this section, we would like to clarify that although this paper considers three different general classes of compressors, it is not this paper's goal to study which specific compressor or general class of compressors has better performance. Moreover, although this paper proposes three communication-efficient distributed algorithms, it is not this paper's goal either to investigate which algorithm has better performance.

\section{Simulations}\label{nonconvex:sec-simulation}
In this section, we verify and illustrate the theoretical results through numerical simulations.
We consider the nonconvex distributed binary classification problem as studied in \cite{sun2019distributed,yi2021linear,xin2021improved}, which is formulated as the optimization problem \eqref{nonconvex:eqn:xopt} with each component function $f_i$ being given by
\begin{align*}
f_i(x)&=\frac{n}{m}\sum_{l=1}^{m_i}\Big((1-y_{il})\log\Big(1+e^{x^\top z_{il}}\Big)
+y_{il}\log\Big(1+e^{-x^\top z_{il}}\Big)\Big)+\sum_{s=1}^{d}\frac{\lambda\mu [x]_s^2}{1+\mu [x]_s^2},
\end{align*}
where $m=\sum_{i=1}^{n}m_i$, $m_i$ is the number of observations held privately by agent $i$, $z_{il}\in\mathbb{R}^p$ is the $l$-th observation with label  $y_{il}\in\{0,1\}$ owned by agent $i$, $\lambda$ and $\mu$ are regularization parameters, and $[x]_s$ is the $s$-th coordinate of $x\in \mathbb{R}^d$. All settings for cost functions and the communication graph are the same as those described in \cite{sun2019distributed,yi2021linear}. Specifically, $n=20$, $d=50$, $m_i=200$, $\lambda=0.001$, and $\mu=1$. The graph used in the simulation is the random geometric graph and the graph parameter is set to be $0.5$. We independently and randomly generate $m$ data points.

We consider the following five compressors:
\begin{itemize}
  \item Unbiased $l$-bits quantizer \cite{liu2021linear}
      \begin{align*}
      \mathcal{C}_1(x)=\frac{\|x\|_\infty}{2^{l-1}}\sign(x)\circ
      \bigg\lfloor\frac{2^{l-1}\abs{x}}{\|x\|_\infty}+\varpi\bigg\rfloor,
      \end{align*}
      where $\sign(\cdot)$, $\abs{\cdot}$, and $\lfloor\cdot\rfloor$ are the element-wise sign, absolute, and floor functions, respectively, $\circ$ denotes the Hadamard product, and $\varpi$ is a random perturbation vector uniformly sampled from $[0,1]^d$. This compressor is unbiased and satisfies Assumption~\ref{nonconvex:ass:compression} with $r=1+r_1$, $\varphi=1/(1+r_1)$, and $r_1=d/4^{l}$. As pointed out in \cite{zhang2021innovation}, transmitting $\mathcal{C}_1(x)$ needs $(l+1)d+b_1$ bits if a scalar can be transmitted with $b_1$ bits with sufficient precision, since only $\|x\|_\infty$, $\sign(x)$, and the positive integer in the bracket need to be transmitted. In this section, we choose $l=2$ and $b_1=64$.

  \item Greedy (Top-$k$) sparsifier \cite{beznosikov2020biased}
  \begin{align*}
  \mathcal{C}_2(x)=\sum_{s=1}^{k}[x]_{i_s}{\bf e}_{i_s},
  \end{align*}
  where $\{{\bf e}_1,\dots,{\bf e}_d\}$ is the standard basis of $\mathbb{R}^d$ and $i_1,\dots,i_k$ are the indices of largest $k$ coordinates in magnitude of $x$. This compressor is biased but contractive. Moreover, it satisfies Assumption~\ref{nonconvex:ass:compression} with $r=1$ and $\varphi=k/d$. Therefore, it also satisfies Assumption~\ref{nonconvex:ass:compression_determin} with $p\ge2$ and $\varphi=k/d$. Transmitting $\mathcal{C}_2(x)$ needs $kb_1$ bits since only $k$ scalars need to be transmitted. In this section, we choose $k=10$.

  \item Norm-sign compressor
  \begin{align*}
  \mathcal{C}_3(x)=\frac{\|x\|_\infty}{2}\sign(x).
  \end{align*}
  This compressor is biased and non-contractive, but satisfies Assumption~\ref{nonconvex:ass:compression} with $r=d/2$ and $\varphi=1/d^2$, see \cite{li2021compressed}. It it also satisfies Assumption~\ref{nonconvex:ass:compression_determin} with $p=\infty$ and $\varphi=0.5$. Transmitting $\mathcal{C}_3(x)$ needs $2d+b_1$ bits since only $\|x\|_\infty$ and $\sign(x)$ need to be transmitted.

  \item Standard uniform quantizer
  \begin{align*}
  \mathcal{C}_4(x)=\Delta\Big\lfloor\frac{x}{\Delta}+\frac{{\bf 1}_d}{2}\Big\rfloor,
  \end{align*}
  where $\Delta$ is a positive integer. This compressor satisfies Assumption~\ref{nonconvex:ass:compression_quantization} with $p=\infty$ and $C=\Delta^2/4$. Therefore, it also satisfies Assumption~\ref{nonconvex:ass:compression_determin} with $p=\infty$ and $\varphi=0.5$ when $\Delta=1$. Transmitting $\mathcal{C}_4(x)$ needs $db_2$ bits if $b_2$ bits are allocated to transmit an integer. In this section, we choose $\Delta=1$ and $b_2=8$.

  \item 1-bit binary quantizer \cite{zhang2021innovation}
  \begin{align*}
  \mathcal{C}_5(x)=\col(Q_1([x]_1),\dots,Q_1([x]_d)),
  \end{align*}
  where $Q_1([x]_s)=0.5$ for $[x]_s\ge0$ and $Q_1([x]_s)=-0.5$ otherwise.
  This compressor satisfies Assumption~\ref{nonconvex:ass:compression_determin} with $p=\infty$ and $\varphi=0.5$. Transmitting $\mathcal{C}_5(x)$ needs $d$ bits since for each  coordinate only two symbols needs to be transmitted.
\end{itemize}

We implement Algorithm~\ref{nonconvex:algorithm-pdgd} using $\mathcal{C}_1$--$\mathcal{C}_3$, Algorithm~\ref{nonconvex:algorithm-pdgd_ef} using $\mathcal{C}_2$ and $\mathcal{C}_3$, and Algorithm~\ref{nonconvex:algorithm-pdgd_determin} using $\mathcal{C}_2$--$\mathcal{C}_5$.
Note that to the best of our knowledge in the literature there are no other similar communication-efficient distributed algorithms for distributed nonconvex optimization as ours. Therefore, we only compare the proposed communication-efficient distributed algorithms with their uncompressed counterpart, i.e., the distributed primal--dual algorithm  \eqref{nonconvex:yi-tac-alg}, which is denoted as DPDA. It is straightforward to see that each agent sends $db_1$ bits per iteration when implementing DPDA. All the hyper-parameters used in the experiment are tuned manually and given in TABLE~\ref{tab:para}.

\bgroup
\def\arraystretch{1.7}
\begin{table*}[ht!]
	\caption{Parameter settings for different algorithm and compressor combinations.}
	\label{tab:para}
	\centering
	\small
	\begin{tabular}{M{2.1cm}|M{2.1cm}|M{1.3cm}|M{1.3cm}|M{1.3cm}|M{1.3cm}|M{1.3cm}|M{1.3cm}|M{1.3cm}N}
		\hline
		Algorithm  &Compressor &$\alpha$ &$\beta$ &$\eta$ &$\psi$ &$\sigma$ &$s_0$ &$\gamma$ &\\
		
		\hline
		DPDA      &---  &85 &5 &1.4 &--- &--- &--- &---  &\\
		
		\hline
		Algorithm~\ref{nonconvex:algorithm-pdgd}    & $\mathcal{C}_1$     &85 &5 &1.4 &0.2 &--- &--- &---  & \\
		
		\hline
		Algorithm~\ref{nonconvex:algorithm-pdgd}   & $\mathcal{C}_2$      &85 &5 &1.4 &0.05 &--- &--- &---   &\\
		
		\hline
		Algorithm~\ref{nonconvex:algorithm-pdgd}      & $\mathcal{C}_3$  &85 &5 &1.3 &0.05 &--- &--- &---  &\\
		
		\hline
		Algorithm~\ref{nonconvex:algorithm-pdgd_ef}        & $\mathcal{C}_2$  &85 &5 &1.4 &0.05 &0.03 &--- &---   &\\
		
		\hline
		Algorithm~\ref{nonconvex:algorithm-pdgd_ef}     & $\mathcal{C}_3$  &85 &5 &1.3 &0.05 &0.03 &--- &---   &\\
		
		\hline
		Algorithm~\ref{nonconvex:algorithm-pdgd_determin}      & $\mathcal{C}_2$  &85 &5 &0.46 &--- &---  &1 &0.99  &\\
		
		\hline
		Algorithm~\ref{nonconvex:algorithm-pdgd_determin}       & $\mathcal{C}_3$  &85 &5 &0.64 &--- &---  &1 &0.99 &\\
		
		\hline
		Algorithm~\ref{nonconvex:algorithm-pdgd_determin}       & $\mathcal{C}_4$  &85 &5 &0.46 &--- &---  &0.01 &0.99 &\\
		
		\hline
		Algorithm~\ref{nonconvex:algorithm-pdgd_determin}       & $\mathcal{C}_5$  &85 &5 &0.46 &--- &---  &1 &0.99 &\\
		
		\hline
	\end{tabular}
\end{table*}
\egroup

\begin{figure}[!ht]
\centering
  \includegraphics[width=0.8\textwidth]{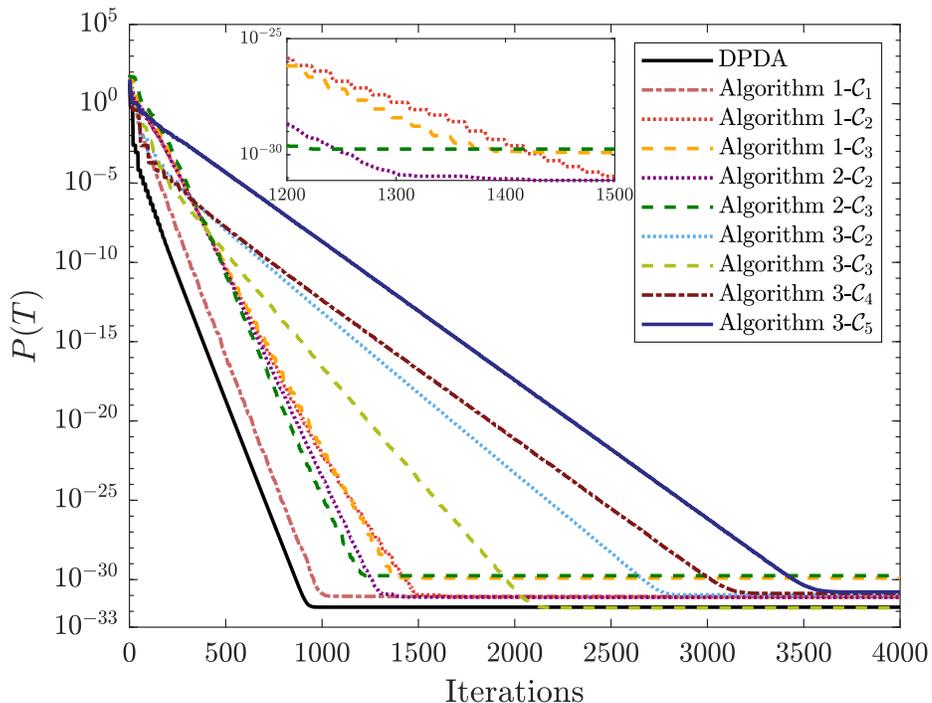}
  \caption{Evolutions of $P(T)$ with respect to the number of iterations.}
  \label{nonconvex:fig:iterations}
\end{figure}

\begin{figure}[!ht]
\centering
  \includegraphics[width=0.8\textwidth]{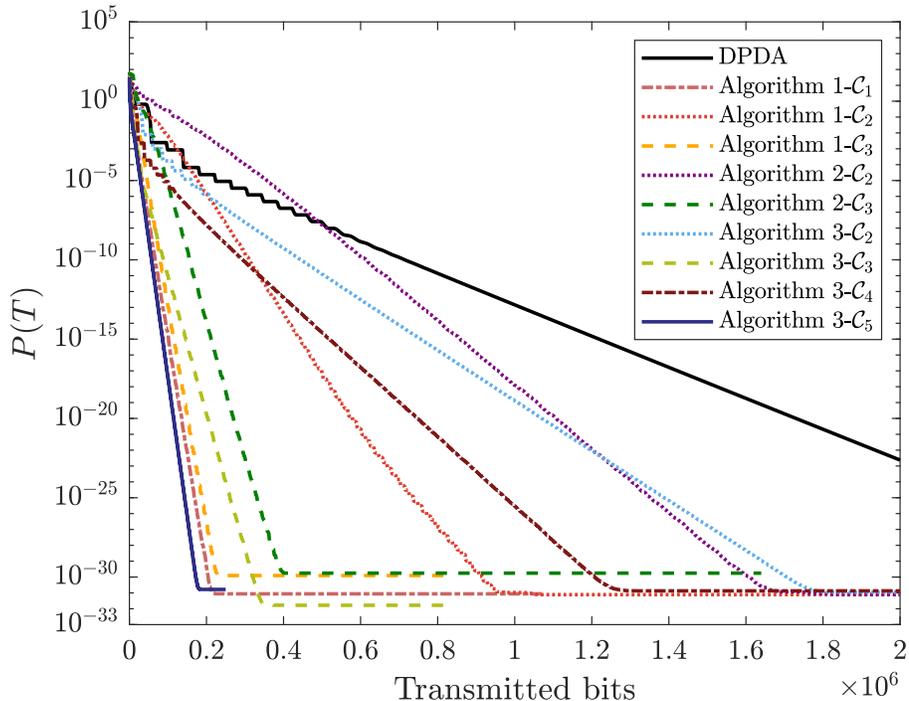}
  \caption{Evolutions of $P(T)$ with respect to the number of transmitted bits.}
  \label{nonconvex:fig:bits}
\end{figure}

\begin{figure}[!ht]
\centering
  \includegraphics[width=0.8\textwidth]{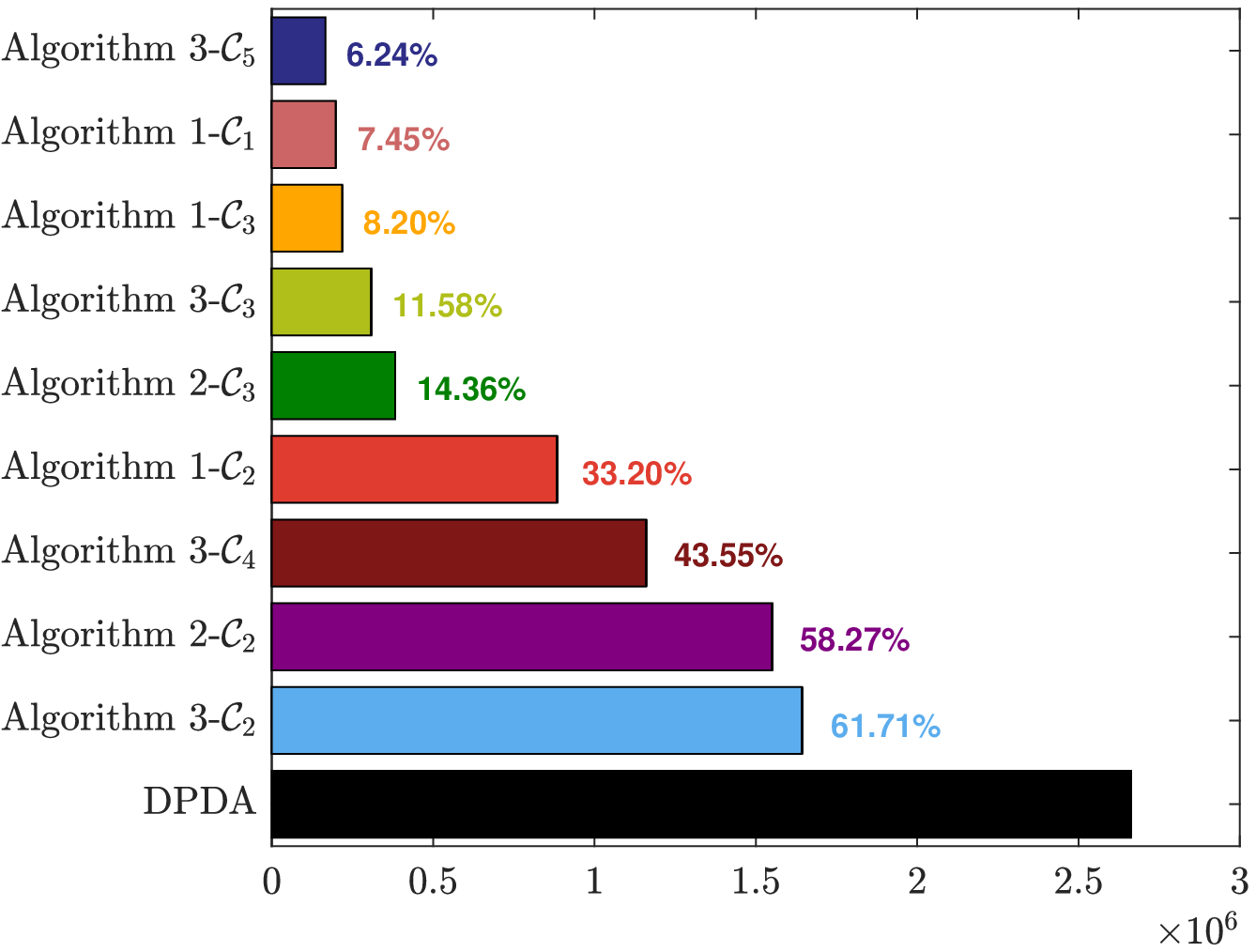}
  \caption{Transmitted bits for different algorithm and compressor combinations to reach $P(T)\le 10^{-30}$.}
  \label{nonconvex:fig:bits_final}
\end{figure}

We use
\begin{align*}
P(T)=\min_{k\in [T]}\Big\{\|\nabla f(\bar{x}_k)\|^ {2}+\frac{1}{n}\sum_{i=1}^{n}\|x_{i,k}-\bar{x}_k\|^ {2}\Big\}
\end{align*}
to measure the performance of each algorithm.
We plot the convergence of $P(T)$ with respect to both number of iterations and bits transmitted between two neighbor agents for the above algorithm and compressor combinations with the same initial condition, as shown in Fig.~\ref{nonconvex:fig:iterations} and Fig.~\ref{nonconvex:fig:bits}, respectively. Moreover, the comparison of transmitted bits for different algorithm and compressor combinations to reach $P(T)\le 10^{-30}$ is provided in Fig.~\ref{nonconvex:fig:bits_final}. We highlight the following observations:
\begin{itemize}
  \item From Fig.~\ref{nonconvex:fig:iterations} we can see that all of the algorithm and compressor combinations have comparable convergence speeds as the corresponding algorithm with exact communication, i.e., DPDA, which is consistent with our theoretical results. Especially, Algorithm~\ref{nonconvex:algorithm-pdgd}-$\mathcal{C}_1$ and DPDA have almost the same convergence speed .

  \item 
      From Fig.~\ref{nonconvex:fig:iterations} we can also see that Algorithm~\ref{nonconvex:algorithm-pdgd}-$\mathcal{C}_3$ (Algorithm~\ref{nonconvex:algorithm-pdgd_ef}-$\mathcal{C}_3$) has almost the same convergence speed as Algorithm~\ref{nonconvex:algorithm-pdgd}-$\mathcal{C}_2$ (Algorithm~\ref{nonconvex:algorithm-pdgd_ef}-$\mathcal{C}_2$), and Algorithm~\ref{nonconvex:algorithm-pdgd_determin}-$\mathcal{C}_3$ has faster convergence speed than Algorithm~\ref{nonconvex:algorithm-pdgd_determin}-$\mathcal{C}_2$. Therefore, non-contractive compressors, e.g., $\mathcal{C}_3$, can converge faster than contractive compressors, e.g., $\mathcal{C}_2$.

  \item 
      From the zoomed figure in Fig.~\ref{nonconvex:fig:iterations} we can see that the error feedback method Algorithm~\ref{nonconvex:algorithm-pdgd_ef}-$\mathcal{C}_2$ (Algorithm~\ref{nonconvex:algorithm-pdgd_ef}-$\mathcal{C}_3$) has faster convergence speed than Algorithm~\ref{nonconvex:algorithm-pdgd}-$\mathcal{C}_2$
      (Algorithm~\ref{nonconvex:algorithm-pdgd}-$\mathcal{C}_3$), which demonstrates the benefit of using the error feedback to correct the bias induced by the biased compressors.

  \item  From Fig.~\ref{nonconvex:fig:bits} we can see that our communication-efficient algorithms converge faster than their exact-communication counterpart when comparing their performances based on the number of bits that each agents communicates, which shows the effectiveness of our proposed algorithms. Especially, Algorithm~\ref{nonconvex:algorithm-pdgd_determin}-$\mathcal{C}_5$, Algorithm~\ref{nonconvex:algorithm-pdgd}-$\mathcal{C}_1$, Algorithm~\ref{nonconvex:algorithm-pdgd}-$\mathcal{C}_3$, Algorithm~\ref{nonconvex:algorithm-pdgd_determin}-$\mathcal{C}_3$, and Algorithm~\ref{nonconvex:algorithm-pdgd_ef}-$\mathcal{C}_3$ converge significantly faster than DPDA. For example, it is illustrated in Fig.~\ref{nonconvex:fig:bits_final} that Algorithm~\ref{nonconvex:algorithm-pdgd_determin}-$\mathcal{C}_5$ only needs $6.24\%$ of the bits used by DPDA to reach a specific level of error.

  \item  From Fig.~\ref{nonconvex:fig:bits} we can also see that Algorithm~\ref{nonconvex:algorithm-pdgd}-$\mathcal{C}_2$ (Algorithm~\ref{nonconvex:algorithm-pdgd}-$\mathcal{C}_3$) converges faster than its error feedback version, i.e., Algorithm~\ref{nonconvex:algorithm-pdgd_ef}-$\mathcal{C}_2$ (Algorithm~\ref{nonconvex:algorithm-pdgd_ef}-$\mathcal{C}_3$) when comparing their performances based on the number of transmitted bits, which reveals the drawback of using the error feedback to correct the bias induced by the biased compressors.

\end{itemize}

\section{Conclusions}\label{nonconvex:sec-conclusion}
In this paper, we studied communication compression for distributed nonconvex optimization. We used three general classes of compressors to design three communication-efficient distributed primal--dual algorithms. We showed that the proposed algorithms can achieve comparable convergence results to state-of-the-art algorithms although the communication is compressed. Interesting directions for future work include considering more general network topologies, reducing communication complexity through periodic communication,  exploring rules for choosing an appropriate compressor for high efficiency, and studying how the important parameters, such as compressor parameters, Lipschitz constant, network connectivity, affect the convergence rate.

\appendix

\subsection{Useful Lemmas}\label{zero:app-lemmas}

The following results are used in the proofs.
\begin{lemma}\label{nonconvex:lemma:pnorm}
	(Equation (5.4.21) on page 333 in \cite{horn2012matrix}.)
	For any $x\in\mathbb{R}^d$, it holds that $\|x\|_p\le\hat{d}\|x\|$ and $\|x\|\le \tilde{d}\|x\|_p$, where $\hat{d}=d^{\frac{1}{p}-\frac{1}{2}}$ and $\tilde{d}=1$ when $p\in[1,2]$, and $\hat{d}=1$ and $\tilde{d}=d^{\frac{1}{2}-\frac{1}{p}}$ when $p>2$.
\end{lemma}

\begin{lemma}\label{nonconvex:lemma:inequality-arithmetic}
	For any $x,y\in\mathbb{R}^d$ and $a,b>0$ satisfying $ab=\frac{1}{4}$, it holds that \begin{align}\label{nonconvex:lemma:inequality-arithmetic-equ}
		x\top y\le a\|x\|^2+b\|y\|^2.
	\end{align}
\end{lemma}
This lemma is a direct extension of the Cauchy--Schwarz inequality.

\begin{lemma}\label{nonconvex:lemma-Xinlei}
	Let $L$ be the Laplacian matrix of an undirected and connected graph $\mathcal{G}$ with $n$ agents and $K_n={\bf I}_n-\frac{1}{n}{\bf 1}_n{\bf 1}^{\top}_n$.
	Then $L$ and $K_n$ are positive semi-definite, $\nullrank(L)=\nullrank(K_n)=\{{\bf 1}_n\}$, $L\le\rho(L){\bf I}_n$, $\rho(K_n)=1$,
	\begin{subequations}
		\begin{align}
			&K_nL=LK_n=L,\label{nonconvex:KL-L-eq}\\
			&0\le\rho_2(L)K_n\le L\le\rho(L)K_n.\label{nonconvex:KL-L-eq2}
		\end{align}
	\end{subequations}
	Moreover, there exists an orthogonal matrix $[r \ R]\in \mathbb{R}^{n \times n}$ with $r=\frac{1}{\sqrt{n}}\mathbf{1}_n$ and $R \in \mathbb{R}^{n\times (n-1)}$ such that
	\begin{subequations}
		\begin{align}
			&L=\left[\begin{array}{ll}r&R
			\end{array}\right]\left[\begin{array}{ll}0&0\\
				0&\Lambda_1
			\end{array}\right]\left[\begin{array}{l}r^\top\\
				R^\top
			\end{array}\right],\label{cdc18so:lemma:LReq}\\
			&PL=LP=K_n,\label{nonconvex:lemma-eq3}\\
			&\rho^{-1}(L){\bf I}_n\leq P\le\rho_2^{-1}(L){\bf I}_n,\label{nonconvex:lemma-eq5}
		\end{align}
	\end{subequations}
	where $\Lambda_1=\diag([\lambda_2,\dots,\lambda_n])$ with $0<\lambda_2\leq\dots\leq\lambda_n$ being the  nonzero eigenvalues of the Laplacian matrix $L$, and
	\begin{align*}
		&P=\left[\begin{array}{ll}r&R
		\end{array}\right]\left[\begin{array}{ll}\lambda_n^{-1}&0\\
			0&\Lambda_1^{-1}
		\end{array}\right]\left[\begin{array}{l}r^\top\\
			R^\top
		\end{array}\right].
	\end{align*}
\end{lemma}
\begin{proof}
	From Lemmas~1 and 2 in the arXiv version of \cite{Yi2018distributed}, we know that all the results except \eqref{nonconvex:lemma-eq3}--\eqref{nonconvex:lemma-eq5} hold.
	
	From that $[r \ R]$ is an orthogonal matrix, \eqref{cdc18so:lemma:LReq}, and the definitions of $K_n$, $P$, and $Q$, it is straightforward to check that \eqref{nonconvex:lemma-eq3}--\eqref{nonconvex:lemma-eq5} holds.
\end{proof}

\subsection{Proof of Theorem~\ref{nonconvex:thm-sm}}\label{nonconvex:proof-thm-sm}
To prove Theorem~\ref{nonconvex:thm-sm}, we first introduce some constants and notations. Denote the following constants
\begin{align*}
	\kappa_1&\ge\max\Big\{\frac{9+\kappa_4}{2\rho_2(L)},~1\Big\},~
	\kappa_2=\max\{\kappa_5,~\sqrt{\kappa_6}\},\\
	\kappa_3&=\min\Big\{\frac{\epsilon_1}{\epsilon_2},~\frac{\epsilon_3}{\epsilon_4},
	~\frac{\epsilon_5}{\epsilon_6},
	~\frac{\sqrt{\epsilon_8^2+4\epsilon_7\epsilon_9}-\epsilon_8}{2\epsilon_9}\Big\},\\
	\kappa_4&>0,~\kappa_5=\max\Big\{\frac{4+5L_f^2}{\kappa_4},~\frac{6}{\rho_2(L)}\Big\},\\
	\kappa_6&=\frac{8(\kappa_1+1)^2L_f^2}{\kappa_5\rho_2(L)}+\frac{4L_f^2}{\rho_2^2(L)},\\
	\epsilon_1&=\frac{\alpha}{2}\rho_2(L)-\frac{1}{4}(9\beta+4+5L_f^2),\\
	\epsilon_2&=3L_f^2+\frac{4(2+\varphi\psi r)}{\varphi\psi r}(\alpha^2\rho^2(L)+L_f^2)+2\beta^2+1+3\alpha^2\rho^2(L),\\
	\epsilon_3&=\frac{\beta}{2}-3\rho^{-1}_2(L),\\
	\epsilon_4&=2\beta^2\rho(L)+\frac{4(2+\varphi\psi r)\beta^2\rho^2(L)}{\varphi\psi r}+\rho_2^{-1}(L),\\
	\epsilon_5&=\frac{1}{8}-\frac{(\alpha+\beta)^2L_f^2}{\beta^5\rho_2(L)}
	-\frac{L_f^2}{2\beta^2\rho^2_2(L)},\\
	\epsilon_6&=\frac{(\alpha+\beta)L_f^2}{2\beta^3\rho_2(L)}+\frac{3L_f^2}{4}
	+\frac{L_f^2}{2\beta^2\rho^2_2(L)}+\frac{L_f}{2},\\
	\epsilon_7&=\frac{\varphi\psi r}{2}(1+\varphi\psi r),\\
	\epsilon_8&=\frac{1}{2}(\alpha+2\beta)\rho(L)r_0+2\beta r_0,\\
	\epsilon_9&=\frac{(8+7\varphi\psi r)\alpha^2\rho^2(L)r_0}{\varphi\psi r}+(2\beta^2+1)r_0,\\
	\epsilon_{10}&=\frac{\alpha\rho_2(L)-\beta}{2\alpha\rho_2(L)},\\
	c_1&=\Big(\frac{(\alpha+\beta)^2}{\eta\beta^5}+\frac{\alpha+\beta}{2\beta^3}\Big)
	\frac{1}{\rho_2(L)}+\frac{1}{2},\\
	c_2&=\frac{\eta+1}{2\eta\beta^2\rho^2_2(L)}+\frac{1}{4},~c_3=\frac{\varphi\psi r}{2},\\
	c_4&=c_3+2c_3^2-4\eta^2(1+c_3^{-1})\alpha^2\rho^2(L)r_0,\\
	c_5&=\frac{\alpha}{2}\rho_2(L)-\frac{1}{4}(\beta+4+5L_f^2),\\
	c_6&=3L_f^2+4(1+c_3^{-1})(\alpha^2\rho^2(L)+L_f^2),\\
	c_7&=\beta^2+\frac{1}{2}+\frac{3\alpha^2\rho^2(L)}{2}.
\end{align*}

Denote $\bsx=\col(x_1,\dots,x_n)$, $\tilde{f}(\bsx)=\sum_{i=1}^{n}f_i(x_i)$, $\bsL=L\otimes {\bf I}_d$, $\bsH=\frac{1}{n}({\bf 1}_n{\bf 1}_n^\top\otimes{\bf I}_d)$, $\bsK=K_n\otimes {\bf I}_d={\bf I}_{nd}-\bsH$, $\bsP=P\otimes {\bf I}_{d}$, $\bar{\bsx}_k={\bf 1}_n\otimes\bar{x}_k=\bsH\bsx_k$, $\bsg_k=\nabla\tilde{f}(\bsx_k)$, $\bar{\bsg}_k=\bsH\bsg_{k}$, $\bsg^0_k=\nabla\tilde{f}(\bar{\bsx}_k)$, $\bar{\bsg}_k^0=\bsH\bsg^0_{k}={\bf 1}_n\otimes\nabla f(\bar{x}_k)$. Moreover, without ambiguity, we denote $\mathcal{C}(\bsx)=\col(\mathcal{C}(x_1),\dots,\mathcal{C}(x_n))$.
We also denote
\begin{align*}
	U_{1,k}&=\frac{1}{2}\|\bm{x}_k \|^2_{\bsK},~U_{2,k}=\frac{1}{2}\Big\|\bsv_k
	+\frac{1}{\beta}\bsg_k^0\Big\|^2_{\frac{\alpha+\beta}{\beta}\bsP}\\
	U_{3,k}&=\bsx_k^\top\bsK\bsP\Big(\bm{v}_k+\frac{1}{\beta}\bsg_k^0\Big),~
	U_{4,k}=n(f(\bar{x}_k)-f^*),\\
	U_{k}&=\sum_{i=1}^{4}U_{i,k},~V_{k}=U_{k}+\|\bsx_{k}-\bsa_{k}\|^2,\\
	\hat{U}_k&=\|\bm{x}_k\|^2_{\bsK}+\Big\|\bsv_k
	+\frac{1}{\beta}\bsg_k^0\Big\|^2_{\bsP}+n(f(\bar{x}_k)-f^*),\\
	\hat{V}_k&=\hat{U}_k+\|\bsx_{k}-\bsa_{k}\|^2,\\
	\bsM_1&=\frac{\alpha}{2}\bsL-\frac{1}{4}(\beta+4+5L_f^2)\bsK,\\
	\bsM_2&=\Big(\beta^2+\frac{1}{2}\Big)\bsK+\frac{\beta(\beta-\alpha)}{2}\bsL+\frac{3\alpha^2}{2}\bsL^2.
\end{align*}

Note that $U_{4,k}$ is well defined since $f^*>-\infty$ as assumed in Assumption~\ref{nonconvex:ass:optset}.
To prove Theorem~\ref{nonconvex:thm-sm}, the following lemma is used, which presents a general relation between two consecutive outputs of Algorithm~\ref{nonconvex:algorithm-pdgd}.
\begin{lemma}\label{noncovex:lemma:pdgd}
	Suppose Assumptions~\ref{nonconvex:ass:compression} and \ref{nonconvex:ass:graph}--\ref{nonconvex:ass:fiu} hold. Let $\{x_{i,k}\}$ be the sequence generated by Algorithm~\ref{nonconvex:algorithm-pdgd} with $\alpha\ge\beta$ and $\psi\in(0,1/r]$. Then,
	\begin{align}\label{nonconvex:vkLya-sm2}
		\mathbf{E}_{\mathcal{C}}[V_{k+1}]&\le  \mathbf{E}_{\mathcal{C}}\Big[V_{k}-\frac{\eta}{4}\|\bar{\bsg}_{k}^0\|^2
		-\|\bsx_k\|^2_{\eta(\epsilon_1-\eta \epsilon_2)\bsK}\nonumber\\
		&\quad-\Big\|\bm{v}_k+\frac{1}{\beta}\bsg_k^0\Big\|^2_{\eta(\epsilon_3-\eta\epsilon_4)\bsP}
		-\eta(\epsilon_5-\eta\epsilon_6)\|\bar{\bsg}_{k}\|^2\nonumber\\
		&\quad-(\epsilon_7-\eta\epsilon_8-\eta^2\epsilon_9)\|\bsx_{k}-\bsa_{k}\|^2\Big].
	\end{align}
\end{lemma}
\begin{proof}
	{\bf (i)}  We first introduce some useful equations.

	The compact form of \eqref{nonconvex:kia-algo-dc-a}, \eqref{nonconvex:kia-algo-dc-b}, \eqref{nonconvex:kia-algo-dc-x-compress}, \eqref{nonconvex:kia-algo-dc-v-compress}, and \eqref{nonconvex:kia-algo-dc-q} is
	\begin{subequations}\label{nonconvex:alg-compress-compact}
		\begin{align}
			\bsa_{k+1}&=\bsa_{k}+\psi\bsq_k,\label{nonconvex:kia-algo-dc-compact-a}\\
			\bsb_{k+1}&=\bsb_{k}+\psi({\bf I}_{np}-\bsL)\bsq_k,\label{nonconvex:alg-compress-compact-b}\\
			\bm{x}_{k+1}&=\bm{x}_k-\eta(\alpha\bsL\hat{\bsx}_k+\beta\bm{v}_k+\nabla \tilde{f}(\bm{x}_k)),\label{nonconvex:kia-algo-dc-compact-x}\\
			\bm{v}_{k+1}&=\bm{v}_k+\eta\beta\bsL\hat{\bsx}_k,\label{nonconvex:kia-algo-dc-compact-v}\\
			\bsq_{k+1}&=\mathcal{C}(\bsx_{k+1}-\bsa_{k+1}).\label{nonconvex:kia-algo-dc-compact-q}
		\end{align}
	\end{subequations}
	
	Denote $\bar{v}_k=\frac{1}{n}({\bf 1}_n^\top\otimes{\bf I}_p)\bsv_k$. Then,
	from \eqref{nonconvex:kia-algo-dc-compact-v} and $\sum_{i=1}^{n}L_{ij}=0$, we know that
	$\bar{v}_{k+1}=\bar{v}_k$.
	This together with the fact that $\sum_{i=1}^{n}v_{i,0}={\bf0}_d$ implies
	\begin{align}
		\bar{v}_k={\bm 0}_d.\label{nonconvex:vkn}
	\end{align}
	Then, from \eqref{nonconvex:vkn} and \eqref{nonconvex:kia-algo-dc-compact-x}, we know that
	\begin{align}
		&\bar{\bsx}_{k+1}=\bar{\bsx}_{k}-\eta\bar{\bsg}_k.\label{nonconvex:xbardynamic}
	\end{align}

	Noting that $\nabla\tilde{f}$ is Lipschitz-continuous with constant $L_{f}>0$ as assumed in Assumption~\ref{nonconvex:ass:fiu}, we have
	\begin{align}
		\|\bsg^0_{k}-\bsg_{k}\|^2\le L_f^2\|\bar{\bsx}_{k}-\bsx_{k}\|^2=L_f^2\|\bsx_{k}\|^2_{\bsK}.\label{nonconvex:gg1}
	\end{align}
	Then, from \eqref{nonconvex:gg1} and $\rho(\bsH)=1$, we have
	\begin{align}
		&\|\bar{\bsg}^0_k-\bar{\bsg}_k\|^2=\|\bsH(\bsg^0_{k}-\bsg_{k})\|^2\le\|\bsg^0_{k}-\bsg_{k}\|^2\le L_f^2\|\bsx_k\|^2_{\bsK}.\label{nonconvex:gg2}
	\end{align}
	
	From $\nabla\tilde{f}$ is Lipschitz-continuous and \eqref{nonconvex:xbardynamic}, we have
	\begin{align}
		&\|\bsg^0_{k+1}-\bsg^0_{k}\|^2\le L_f^2\|\bar{\bsx}_{k+1}-\bar{\bsx}_{k}\|^2=\eta^2L_f^2\|\bar{\bsg}_k\|^2.\label{nonconvex:gg}
	\end{align}

	From Lemma~1.2.3 in \cite{nesterov2018lectures}, we know that \eqref{nonconvex:smooth} implies
	\begin{align}
		&|f_i(y)-f_i(x)-(y-x)^\top\nabla f_i(x)|\le\frac{L_f}{2}\|y-x\|^2,~\forall x,y\in\mathbb{R}^{d}. \label{nonconvex:lemma:smooth}
	\end{align}
	From \eqref{nonconvex:lemma:smooth} and \eqref{nonconvex:xbardynamic}, we have
	\begin{align}\label{nonconvex:lemma:smooth2}
		\tilde{f}(\bar{\bsx}_{k+1})-\tilde{f}(\bar{\bsx}_k)
		\le-\eta\bar{\bsg}_{k}^\top\bsg^0_k
		+\frac{\eta^2L_f}{2}\|\bar{\bsg}_{k}\|^2.
	\end{align}
	
	\noindent {\bf (ii)}  This step is to show the relation between $U_{1,k+1}$ and $U_{1,k}$. We have
	\begin{align}
		U_{1,k+1}&=\frac{1}{2}\|\bm{x}_{k+1} \|^2_{\bsK}
		=\frac{1}{2}\|\bm{x}_k-\eta(\alpha\bsL\hat{\bm{x}}_k+\beta\bm{v}_k+\bsg_k) \|^2_{\bsK}\nonumber\\
		&=\frac{1}{2}\|\bm{x}_k\|^2_{\bsK}-\eta\alpha\bsx^\top_k\bsL\hat{\bm{x}}_k
		+\|\hat{\bm{x}}_k\|^2_{\frac{\eta^2\alpha^2}{2}\bsL^2}\nonumber\\
		&\quad-\eta\beta(\bsx^\top_k-\eta\alpha\hat{\bm{x}}_k^\top\bsL)\bsK
		\Big(\bm{v}_k+\frac{1}{\beta}\bsg_k\Big)\nonumber
		+\Big\|\bm{v}_k+\frac{1}{\beta}\bsg_k\Big\|^2_{\frac{\eta^2\beta^2}{2}\bsK}\nonumber\\
		&=\frac{1}{2}\|\bm{x}_k\|^2_{\bsK}-\eta\alpha\bsx^\top_k\bsL(\bm{x}_k+\hat{\bm{x}}_k-\bm{x}_k)
		+\|\hat{\bm{x}}_k\|^2_{\frac{\eta^2\alpha^2}{2}\bsL^2}\nonumber\\
		&\quad-\eta\beta(\bsx^\top_k-\eta\alpha\hat{\bm{x}}_k^\top\bsL)\bsK\Big(\bm{v}_k
		+\frac{1}{\beta}\bsg_k^0
		+\frac{1}{\beta}\bsg_k-\frac{1}{\beta}\bsg_k^0\Big)\nonumber\\
		&\quad+\Big\|\bm{v}_k+\frac{1}{\beta}\bsg_k^0
		+\frac{1}{\beta}\bsg_k-\frac{1}{\beta}\bsg_k^0\Big\|^2_{\frac{\eta^2\beta^2}{2}\bsK}\nonumber\\
		&\le\frac{1}{2}\|\bm{x}_k\|^2_{\bsK}-\|\bsx_k\|^2_{\eta\alpha\bsL}
		+\|\bsx_k\|^2_{\frac{\eta\alpha}{2}\bsL}+\|\hat{\bm{x}}_k-\bm{x}_k\|^2_{\frac{\eta\alpha}{2}\bsL}
		\nonumber\\
		&\quad+\|\hat{\bm{x}}_k\|^2_{\frac{\eta^2\alpha^2}{2}\bsL^2}
		-\eta\beta\bsx^\top_k\bsK\Big(\bm{v}_k+\frac{1}{\beta}\bsg_k^0\Big)
		+\frac{\eta}{2}\|\bm{x}_k\|^2_{\bsK}\nonumber\\
		&\quad+\frac{\eta}{2}\|\bsg_k-\bsg_k^0\|^2+\|\hat{\bm{x}}_k\|^2_{\frac{\eta^2\alpha^2}{2}\bsL^2}
		+\frac{\eta^2\beta^2}{2}\Big\|\bm{v}_k+\frac{1}{\beta}\bsg_k^0\Big\|^2\nonumber\\
		&\quad+\|\hat{\bm{x}}_k\|^2_{\frac{\eta^2\alpha^2}{2}\bsL^2}
		+\frac{\eta^2}{2}\|\bsg_k-\bsg_k^0\|^2\nonumber\\
		&\quad+\eta^2\beta^2\Big\|\bm{v}_k+\frac{1}{\beta}\bsg_k^0\Big\|^2
		+\eta^2\|\bsg_k-\bsg_k^0\|^2\nonumber\\
		&=\frac{1}{2}\|\bm{x}_k\|^2_{\bsK}-\|\bsx_k\|^2_{\frac{\eta\alpha}{2}\bsL-\frac{\eta}{2}\bsK}
		+\|\hat{\bm{x}}_k\|^2_{\frac{3\eta^2\alpha^2}{2}\bsL^2}\nonumber\\
		&\quad+\frac{\eta}{2}(1+3\eta)\|\bsg_k-\bsg_k^0\|^2
		+\|\hat{\bm{x}}_k-\bm{x}_k\|^2_{\frac{\eta\alpha}{2}\bsL}\nonumber\\
		&\quad-\eta\beta(\hat{\bsx}_k+\bsx_k-\hat{\bsx}_k)^\top\bsK
		\Big(\bm{v}_k+\frac{1}{\beta}\bsg_k^0\Big)\nonumber+\frac{3\eta^2\beta^2}{2}\Big\|\bm{v}_k+\frac{1}{\beta}\bsg_k^0\Big\|^2\nonumber\\
		&\le\frac{1}{2}\|\bm{x}_k\|^2_{\bsK}-\|\bsx_k\|^2_{\frac{\eta\alpha}{2}\bsL-\frac{\eta}{2}\bsK}
		+\|\hat{\bm{x}}_k\|^2_{\frac{3\eta^2\alpha^2}{2}\bsL^2}\nonumber\\
		&\quad+\frac{\eta}{2}(1+3\eta)\|\bsg_k-\bsg_k^0\|^2
		+\|\hat{\bm{x}}_k-\bm{x}_k\|^2_{\frac{\eta}{2}(\alpha\bsL+2\beta\rho(L)\bsK)}\nonumber\\
		&\quad-\eta\beta\hat{\bsx}^\top_k\bsK\Big(\bm{v}_k+\frac{1}{\beta}\bsg_k^0\Big)\nonumber
		+\frac{6\eta^2\beta^2+\eta\beta\rho^{-1}(L)}{4}\Big\|\bm{v}_k+\frac{1}{\beta}\bsg_k^0\Big\|^2
		\nonumber\\
		&\le\frac{1}{2}\|\bm{x}_k\|^2_{\bsK}-\|\bsx_k\|^2_{\frac{\eta\alpha}{2}\bsL-\frac{\eta}{2}\bsK
			-\frac{\eta}{2}(1+3\eta)L_f^2\bsK}+\|\hat{\bm{x}}_k\|^2_{\frac{3\eta^2\alpha^2}{2}\bsL^2}\nonumber\\
		&\quad-\eta\beta\hat{\bsx}^\top_k\bsK\Big(\bm{v}_k+\frac{1}{\beta}\bsg_k^0\Big)
		+\Big\|\bm{v}_k+\frac{1}{\beta}\bsg_k^0\Big\|^2_{\frac{6\eta^2\beta^2\rho(L)+\eta\beta}{4}\bsP}\nonumber\\
		&\quad+\frac{\eta}{2}(\alpha+2\beta)\rho(L)\|\bsx_k-\hat{\bsx}_k\|^2,\label{nonconvex:v1k}
	\end{align}
	where the second and third equalities hold due to \eqref{nonconvex:kia-algo-dc-compact-x} and \eqref{nonconvex:KL-L-eq}, respectively; the first and second inequalities hold due to \eqref{nonconvex:lemma:inequality-arithmetic-equ} and $\rho(\bsK)=1$; and the last  inequality holds due to \eqref{nonconvex:lemma-eq5} and \eqref{nonconvex:gg1}.
	
	\noindent {\bf (iii)}  This step is to show the relation between $U_{2,k+1}$ and $U_{2,k}$.	We have
	\begin{align}
		U_{2,k+1}&=\frac{1}{2}\Big\|\bsv_{k+1}+\frac{1}{\beta}\bsg_{k+1}^0\Big\|^2_{\frac{\alpha+\beta}{\beta}\bsP}\nonumber\\
		&=\frac{1}{2}\Big\|\bm{v}_k+\frac{1}{\beta}\bsg_{k}^0+\eta\beta\bsL\hat{\bm{x}}_k
		+\frac{1}{\beta}(\bsg_{k+1}^0-\bsg_{k}^0) \Big\|^2_{\bsP+\frac{\alpha}{\beta}\bsP}\nonumber\\
		&=\frac{1}{2}\Big\|\bm{v}_k+\frac{1}{\beta}\bsg_{k}^0\Big\|^2_{\bsP+\frac{\alpha}{\beta}\bsP}
		+\eta(\alpha+\beta)\hat{\bsx}^\top_k\bsK\Big(\bm{v}_k+\frac{1}{\beta}\bsg_k^0\Big)\nonumber\\
		&\quad+\|\hat{\bsx}_k\|^2_{\frac{\eta^2\beta}{2}(\alpha+\beta)\bsL}
		+\frac{1}{2\beta^2}\|\bsg_{k+1}^0-\bsg_{k}^0\|^2_{\bsP+\frac{\alpha}{\beta}\bsP}\nonumber\\
		&\quad+\frac{1}{\beta}\Big(\bm{v}_k+\frac{1}{\beta}\bsg_{k}^0
		\Big)^\top\Big(\bsP+\frac{\alpha}{\beta}\bsP\Big)(\bsg_{k+1}^0-\bsg_{k}^0)\nonumber\\
		&\quad+\eta\hat{\bm{x}}_k^\top\Big(\bsK+\frac{\alpha}{\beta}\bsK\Big)(\bsg_{k+1}^0-\bsg_{k}^0)\nonumber\\
		&\le\frac{1}{2}\Big\|\bm{v}_k+\frac{1}{\beta}\bsg_{k}^0\Big\|^2_{\bsP+\frac{\alpha}{\beta}\bsP}
		+\eta(\alpha+\beta)\hat{\bsx}^\top_k\bsK\Big(\bm{v}_k+\frac{1}{\beta}\bsg_k^0\Big)\nonumber\\
		&\quad+\|\hat{\bsx}\|^2_{\frac{\eta^2\beta}{2}(\alpha+\beta)\bsL}
		+\|\bsg_{k+1}^0-\bsg_{k}^0\|^2_{\frac{\alpha+\beta}{2\beta^3}\bsP}\nonumber\\
		&\quad+\Big\|\bm{v}_k+\frac{1}{\beta}\bsg_{k}^0\Big\|^2_{\frac{\eta\beta}{4}\bsP}
		+\|\bsg_{k+1}^0-\bsg_{k}^0\|^2_{\frac{(\alpha+\beta)^2}{\eta\beta^5}\bsP}\nonumber\\
		&\quad+\|\hat{\bm{x}}_k\|^2_{\frac{\eta^2}{2}\bsK}
		+\frac{1}{2}\|\bsg_{k+1}^0-\bsg_{k}^0\|^2
		+\frac{\eta\alpha}{\beta}\hat{\bm{x}}_k^\top\bsK(\bsg_{k+1}^0-\bsg_{k}^0)\nonumber\\
		&=\frac{1}{2}\Big\|\bm{v}_k+\frac{1}{\beta}\bsg_{k}^0\Big\|^2_{\bsP+\frac{\alpha}{\beta}\bsP}
		+\eta(\alpha+\beta)\hat{\bsx}^\top_k\bsK\Big(\bm{v}_k+\frac{1}{\beta}\bsg_k^0\Big)\nonumber\\
		&\quad+\|\hat{\bsx}\|^2_{\frac{\eta^2\beta}{2}(\alpha+\beta)\bsL+\frac{\eta^2}{2}\bsK}
		+\Big\|\bm{v}_k+\frac{1}{\beta}\bsg_{k}^0\Big\|^2_{\frac{\eta\beta}{4}\bsP}\nonumber\\
		&\quad+\|\bsg_{k+1}^0-\bsg_{k}^0\|^2_{\frac{(\alpha+\beta)^2}{\eta\beta^5}\bsP
			+\frac{\alpha+\beta}{2\beta^3}\bsP}\nonumber\\
		&\quad+\frac{1}{2}\|\bsg_{k+1}^0-\bsg_{k}^0\|^2
		+\frac{\eta\alpha}{\beta}\hat{\bm{x}}_k^\top\bsK(\bsg_{k+1}^0-\bsg_{k}^0)\nonumber\\
		&\le\frac{1}{2}\Big\|\bm{v}_k+\frac{1}{\beta}\bsg_{k}^0\Big\|^2_{\bsP+\frac{\alpha}{\beta}\bsP}
		+\eta(\alpha+\beta)\hat{\bsx}^\top_k\bsK\Big(\bm{v}_k+\frac{1}{\beta}\bsg_k^0\Big)\nonumber\\
		&\quad+\|\hat{\bsx}\|^2_{\frac{\eta^2\beta}{2}(\alpha+\beta)\bsL+\frac{\eta^2}{2}\bsK}
		+\Big\|\bm{v}_k+\frac{1}{\beta}\bsg_{k}^0\Big\|^2_{\frac{\eta\beta}{4}\bsP}\nonumber\\
		&\quad+c_1\|\bsg_{k+1}^0-\bsg_{k}^0\|^2+\frac{\eta\alpha}{\beta}
		\hat{\bm{x}}_k^\top\bsK(\bsg_{k+1}^0-\bsg_{k}^0)
		\nonumber\\
		&\le\frac{1}{2}\Big\|\bm{v}_k+\frac{1}{\beta}\bsg_{k}^0\Big\|^2_{\frac{\alpha+\beta}{\beta}\bsP}
		+\eta(\alpha+\beta)\hat{\bsx}^\top_k\bsK\Big(\bm{v}_k+\frac{1}{\beta}\bsg_k^0\Big)\nonumber\\
		&\quad+\|\hat{\bsx}\|^2_{\frac{\eta^2\beta}{2}(\alpha+\beta)\bsL+\frac{\eta^2}{2}\bsK}
		+\Big\|\bm{v}_k+\frac{1}{\beta}\bsg_{k}^0\Big\|^2_{\frac{\eta\beta}{4}\bsP}\nonumber\\
		&\quad+\eta^2 c_1L_f^2\|\bar{\bsg}_k\|^2+\frac{\eta\alpha}{\beta}\hat{\bm{x}}_k^\top\bsK(\bsg_{k+1}^0-\bsg_{k}^0),\label{nonconvex:v2k}
	\end{align}
	where the second and third equalities hold due to \eqref{nonconvex:kia-algo-dc-compact-v} and \eqref{nonconvex:lemma-eq3}, respectively;  the first, second, and last inequalities hold due to \eqref{nonconvex:lemma:inequality-arithmetic-equ}, \eqref{nonconvex:lemma-eq5} and  \eqref{nonconvex:gg}, respectively.
	
	\noindent {\bf (iv)}  This step is to show the relation between $U_{3,k+1}$ and $U_{3,k}$.	We have
	\begin{align}
		U_{3,k+1}&=\bsx_{k+1}^\top\bsK\bsP\Big(\bm{v}_{k+1}+\frac{1}{\beta}\bsg_{k+1}^0\Big)\nonumber\\
		&=(\bm{x}_k-\eta(\alpha\bsL\hat{\bsx}_k+\beta\bm{v}_k+\bsg_k^0+\bsg_k-\bsg_k^0))^\top
		\bsK\bsP\Big(\bm{v}_k\nonumber\\
		&\quad+\frac{1}{\beta}\bsg_{k}^0+\eta\beta\bsL\hat{\bsx}_k+\frac{1}{\beta}(\bsg_{k+1}^0-\bsg_{k}^0)\Big)\nonumber\\
		&=(\bm{x}_k^\top\bsK\bsP-\eta(\alpha+\eta\beta^2)\hat{\bsx}_k^\top\bsK)
		\Big(\bm{v}_k+\frac{1}{\beta}\bsg_{k}^0\Big)\nonumber\\
		&\quad+\eta\beta\bm{x}_k^\top\bsK\hat{\bsx}_k-\|\hat{\bsx}_k\|^2_{\eta^2\alpha\beta\bsL}\nonumber\\
		&\quad
		+\frac{1}{\beta}(\bm{x}_k^\top\bsK\bsP-\eta\alpha\hat{\bsx}_k^\top\bsK)(\bsg_{k+1}^0-\bsg_{k}^0)\nonumber\\
		&\quad-\eta(\beta\bm{v}_k+\bsg_{k}^0+\bsg_k-\bsg_k^0-\bar{\bsg}_k)^\top\bsP
		\Big(\bm{v}_k+\frac{1}{\beta}\bsg_{k}^0\Big)\nonumber\\
		&\quad
		-\eta\Big(\bm{v}_k+\frac{1}{\beta}\bsg_{k}^0\Big)^\top\bsP\bsK(\bsg_{k+1}^0-\bsg_{k}^0)\nonumber\\
		&\quad-\eta(\bsg_k-\bsg_k^0)^\top
		\Big(\eta\beta\bsK\hat{\bm{x}}_k+\frac{1}{\beta}\bsK\bsP(\bsg_{k+1}^0-\bsg_{k}^0)\Big)\nonumber\\
		&\le(\bm{x}_k^\top\bsK\bsP-\eta\alpha\hat{\bsx}_k^\top\bsK)
		\Big(\bm{v}_k+\frac{1}{\beta}\bsg_{k}^0\Big)
		+\|\hat{\bsx}_k\|^2_{\frac{\eta^2\beta^2}{2}\bsK}\nonumber\\
		&\quad+\frac{\eta^2\beta^2}{2}\Big\|\bm{v}_k+\frac{1}{\beta}\bsg_{k}^0\Big\|^2
		+\|\bm{x}_k\|^2_{\frac{\eta\beta}{4}\bsK}+\|\hat{\bm{x}}_k\|^2_{\eta\beta(\bsK-\eta\alpha\bsL)}\nonumber\\
		&\quad+\|\bm{x}_k\|^2_{\frac{\eta}{2}\bsK}+\|\bsg_{k+1}^0-\bsg_{k}^0\|^2_{\frac{1}{2\eta\beta^2}\bsP^2}
		\nonumber\\
		&\quad-\frac{\eta\alpha}{\beta}\hat{\bm{x}}_k^\top\bsK(\bsg_{k+1}^0-\bsg_{k}^0)
		-\Big\|\bm{v}_k+\frac{1}{\beta}\bsg_{k}^0\Big\|^2_{\eta\beta\bsP}\nonumber\\
		&\quad+\frac{\eta}{4}\|\bsg_k-\bsg_k^0\|^2+\frac{\eta}{8}\|\bar{\bsg}_k\|^2
		+\Big\|\bm{v}_k+\frac{1}{\beta}\bsg_{k}^0\Big\|^2_{3\eta\bsP^2}\nonumber\\
		&\quad+\Big\|\bm{v}_k+\frac{1}{\beta}\bsg_{k}^0\Big\|^2_{\eta^2\bsP^2}
		+\frac{1}{4}\|\bsg_{k+1}^0-\bsg_{k}^0\|^2
		+\frac{\eta^2}{2}\|\bsg_k-\bsg_k^0\|^2\nonumber\\
		&\quad+\|\hat{\bm{x}}_k\|^2_{\frac{\eta^2\beta^2}{2}\bsK}
		+\eta^2\|\bsg_k-\bsg_k^0\|^2
		+\|\bsg_{k+1}^0-\bsg_{k}^0\|^2_{\frac{1}{2\beta^2}\bsP^2}\nonumber\\
		&=(\bm{x}_k^\top\bsK\bsP-\eta\alpha\hat{\bsx}_k^\top\bsK)\Big(\bm{v}_k+\frac{1}{\beta}\bsg_{k}^0\Big)
		+\|\bm{x}_k\|^2_{\frac{\eta(\beta+2)}{4}\bsK}\nonumber\\
		&\quad+\|\hat{\bm{x}}_k\|^2_{\eta\beta\bsK
			+\eta^2(\beta^2\bsK-\alpha\beta\bsL)}
		+\frac{\eta}{4}(1+6\eta)\|\bsg_k-\bsg_k^0\|^2\nonumber\\
		&\quad+\|\bsg_{k+1}^0-\bsg_{k}^0\|^2_{\frac{\eta+1}{2\eta\beta^2}\bsP^2+\frac{1}{4}{\bf I}_{nd}}-\frac{\eta\alpha}{\beta}\hat{\bm{x}}_k^\top\bsK(\bsg_{k+1}^0-\bsg_{k}^0)\nonumber\\
		&\quad+\frac{\eta}{8}\|\bar{\bsg}_k\|^2
		-\Big\|\bm{v}_k+\frac{1}{\beta}\bsg_{k}^0\Big\|^2_{\eta\beta\bsP-3\eta\bsP^2-\eta^2\bsP^2
			-\frac{\eta^2\beta^2}{2}{\bf I}_{np}}\nonumber\\
		&\le(\bm{x}_k^\top\bsK\bsP-\eta\alpha\hat{\bsx}_k^\top\bsK)\Big(\bm{v}_k+\frac{1}{\beta}\bsg_{k}^0\Big)
		+\|\bm{x}_k\|^2_{\frac{\eta(\beta+2)}{4}\bsK}\nonumber\\
		&\quad+\|\hat{\bm{x}}_k\|^2_{\eta\beta\bsK
			+\eta^2(\beta^2\bsK-\alpha\beta\bsL)}
		+\frac{\eta}{4}(1+6\eta)\|\bsg_k-\bsg_k^0\|^2\nonumber\\
		&\quad+c_2\|\bsg_{k+1}^0-\bsg_{k}^0\|^2-\frac{\eta\alpha}{\beta}\hat{\bm{x}}_k^\top
		\bsK(\bsg_{k+1}^0-\bsg_{k}^0)+\frac{\eta}{8}\|\bar{\bsg}_k\|^2\nonumber\\
		&\quad-\Big\|\bm{v}_k+\frac{1}{\beta}\bsg_{k}^0\Big\|^2_{\eta(\beta-3\rho^{-1}_2(L))\bsP
			-\eta^2(\rho^{-1}_2(L)+\frac{\beta^2}{2}\rho(L))\bsP}
		\nonumber\\
		&\le\bm{x}_k^\top\bsK\bsP\Big(\bm{v}_k+\frac{1}{\beta}\bsg_{k}^0\Big)
		-\eta\alpha\hat{\bsx}_k^\top\bsK\Big(\bm{v}_k+\frac{1}{\beta}\bsg_{k}^0\Big)\nonumber\\
		&\quad+\|\hat{\bm{x}}_k\|^2_{\eta\beta\bsK
			+\eta^2(\beta^2\bsK-\alpha\beta\bsL)}
		+\|\bm{x}_k\|^2_{\frac{\eta(\beta+2)}{4}\bsK+\frac{\eta}{4}(1+6\eta)L_f^2\bsK}\nonumber\\
		&\quad-\frac{\eta\alpha}{\beta}\hat{\bm{x}}_k^\top
		\bsK(\bsg_{k+1}^0-\bsg_{k}^0)+\Big(\eta^2c_2L_f^2+\frac{\eta}{8}\Big)\|\bar{\bsg}_k\|^2\nonumber\\
		&\quad-\Big\|\bm{v}_k+\frac{1}{\beta}\bsg_{k}^0\Big\|^2_{\eta(\beta-3\rho^{-1}_2(L))\bsP
			-\eta^2(\rho^{-1}_2(L)+\frac{\beta^2}{2}\rho(L))\bsP},\label{nonconvex:v3k}
	\end{align}
	where the second equality holds due to \eqref{nonconvex:kia-algo-dc-compact-x} and \eqref{nonconvex:kia-algo-dc-compact-v}; the third equality holds due to \eqref{nonconvex:KL-L-eq}, \eqref{nonconvex:lemma-eq3}, \eqref{nonconvex:vkn} and $\bsK={\bf I}_{nd}-\bsH$; the first inequality holds due \eqref{nonconvex:lemma:inequality-arithmetic-equ} and $\rho(\bsK)=1$;  the second inequality holds due to \eqref{nonconvex:lemma-eq5}; and the last  inequality holds due to \eqref{nonconvex:gg1} and \eqref{nonconvex:gg}.
	
	\noindent {\bf (v)}  This step is to show the relation between $U_{4,k+1}$ and $U_{4,k}$.	We have
	\begin{align}
		U_{4,k+1}&=n(f(\bar{x}_{k+1})-f^*)=\tilde{f}(\bar{\bsx}_{k+1})-nf^*\nonumber\\
		&=\tilde{f}(\bar{\bsx}_k)-nf^*+\tilde{f}(\bar{\bsx}_{k+1})-\tilde{f}(\bar{\bsx}_k)\nonumber\\
		&\le\tilde{f}(\bar{\bsx}_k)-nf^*
		-\eta\bar{\bsg}_{k}^\top\bsg^0_k
		+\frac{\eta^2L_f}{2}\|\bar{\bsg}_{k}\|^2\nonumber\\
		&=\tilde{f}(\bar{\bsx}_k)-nf^*
		-\eta\bar{\bsg}_{k}^\top\bar{\bsg}^0_k
		+\frac{\eta^2L_f}{2}\|\bar{\bsg}_{k}\|^2\nonumber\\
		&=n(f(\bar{x}_k)-f^*)
		-\frac{\eta}{2}\bar{\bsg}_{k}^\top(\bar{\bsg}_k+\bar{\bsg}^0_k-\bar{\bsg}_k)-\frac{\eta}{2}(\bar{\bsg}_{k}-\bar{\bsg}^0_k+\bar{\bsg}^0_k)^\top\bar{\bsg}^0_k
		+\frac{\eta^2L_f}{2}\|\bar{\bsg}_{k}\|^2\nonumber\\
		&\le n(f(\bar{x}_k)-f^*)-\frac{\eta}{4}\|\bar{\bsg}_{k}\|^2
		+\frac{\eta}{4}\|\bar{\bsg}^0_k-\bar{\bsg}_k\|^2
		-\frac{\eta}{4}\|\bar{\bsg}_{k}^0\|^2
		+\frac{\eta}{4}\|\bar{\bsg}^0_k-\bar{\bsg}_k\|^2
		+\frac{\eta^2L_f}{2}\|\bar{\bsg}_{k}\|^2\nonumber\\
		&=n(f(\bar{x}_k)-f^*)-\frac{\eta}{4}(1-2\eta L_f)\|\bar{\bsg}_{k}\|^2
		+\frac{\eta}{2}\|\bar{\bsg}^0_k-\bar{\bsg}_k\|^2
		-\frac{\eta}{4}\|\bar{\bsg}_{k}^0\|^2\nonumber\\
		&\le n(f(\bar{x}_k)-f^*)-\frac{\eta}{4}(1-2\eta L_f)\|\bar{\bsg}_{k}\|^2
		+\|\bsx_k\|^2_{\frac{\eta}{2}L_f^2\bsK}
		-\frac{\eta}{4}\|\bar{\bsg}_{k}^0\|^2,\label{nonconvex:v4k-1}
	\end{align}
	where the first inequality holds due to \eqref{nonconvex:lemma:smooth2}; the fourth equality holds due to $\bar{\bsg}_{k}^\top\bsg^0_k=\bsg_{k}^\top\bsH\bsg^0_k=\bsg_{k}^\top\bsH\bsH\bsg^0_k=\bar{\bsg}_{k}^\top\bar{\bsg}^0_k$; the second inequality holds due to \eqref{nonconvex:lemma:inequality-arithmetic-equ}; and the last inequality holds due to \eqref{nonconvex:gg2}.
	
	\noindent {\bf (vi)}  This step is to show the relation between $\|\bsx_{k+1}-\bsa_{k+1}\|^2$ and $\|\bsx_{k}-\bsa_{k}\|^2$.	
	Denote $\mathcal{C}_r(\cdot)=\mathcal{C}(\cdot)/r$, then we have
	\begin{align}
		\mathbf{E}_{\mathcal{C}}[\|\bsx_{k+1}-\bsa_{k+1}\|^2]
		&=\mathbf{E}_{\mathcal{C}}[\|\bsx_{k+1}-\bsx_{k}+\bsx_{k}-\bsa_{k}-\psi\bsq_k\|^2]\nonumber\\
		&=\mathbf{E}_{\mathcal{C}}[\|\bsx_{k+1}-\bsx_{k}+(1-\psi r)(\bsx_{k}-\bsa_{k})\nonumber\\
		&\quad+\psi r(\bsx_{k}-\bsa_{k}-\mathcal{C}_r(\bsx_{k}-\bsa_{k}))\|^2]\nonumber\\
		&\le(1+c_3^{-1})\mathbf{E}_{\mathcal{C}}[\|\bsx_{k+1}-\bsx_{k}\|^2]\nonumber\\
		&\quad+(1+c_3)\mathbf{E}_{\mathcal{C}}[\|(1-\psi r)(\bsx_{k}-\bsa_{k})\nonumber\\
		&\quad+\psi r(\bsx_{k}-\bsa_{k}-\mathcal{C}_r(\bsx_{k}-\bsa_{k}))\|^2]\nonumber\\
		&\le(1+c_3^{-1})\mathbf{E}_{\mathcal{C}}[\|\bsx_{k+1}-\bsx_{k}\|^2]\nonumber\\
		&\quad+(1+c_3)(1-\psi r)\mathbf{E}_{\mathcal{C}}[\|\bsx_{k}-\bsa_{k}\|^2]\nonumber\\
		&\quad+(1+c_3)\psi r\mathbf{E}_{\mathcal{C}}[\|\bsx_{k}-\bsa_{k}-\mathcal{C}_r(\bsx_{k}-\bsa_{k})\|^2]\nonumber\\
		&\le(1+c_3^{-1})\mathbf{E}_{\mathcal{C}}[\|\bsx_{k+1}-\bsx_{k}\|^2]\nonumber\\
		&\quad+(1+c_3)(1-\psi r)\mathbf{E}_{\mathcal{C}}[\|\bsx_{k}-\bsa_{k}\|^2]\nonumber\\
		&\quad+(1+c_3)\psi r(1-\varphi)\mathbf{E}_{\mathcal{C}}[\|\bsx_{k}-\bsa_{k}\|^2]\nonumber\\
		&=(1+c_3^{-1})\mathbf{E}_{\mathcal{C}}[\|\bsx_{k+1}-\bsx_{k}\|^2]\nonumber\\
		&\quad+(1-c_3-2c_3^2)\mathbf{E}_{\mathcal{C}}[\|\bsx_{k}-\bsa_{k}\|^2],
		\label{nonconvex:xminush_compress}
	\end{align}
	where the first and second equalities hold due to \eqref{nonconvex:kia-algo-dc-compact-a} and \eqref{nonconvex:kia-algo-dc-compact-q}, respectively; the first inequality holds due to \eqref{nonconvex:lemma:inequality-arithmetic-equ} and $c_3>0$; the second inequality holds due to \eqref{nonconvex:lemma:inequality-arithmetic-equ} and $\psi r\in(0,1]$; and the last inequality holds due to \eqref{nonconvex:ass:compression_equ_scaling}.
	
	We have
	\begin{align}\label{nonconvex:xkoneminusx}
		\|\bsx_{k+1}-\bsx_k\|^2&=\eta^2\|\alpha\bsL\hat{\bsx}_k+\beta\bm{v}_k+\bsg_k\|^2\nonumber\\
		&=\eta^2\|\alpha\bsL(\hat{\bsx}_k-\bsx_k)+\alpha\bsL\bsx_k+\beta\bm{v}_k
		+\bsg_k^0+\bsg_k-\bsg_k^0\|^2\nonumber\\
		&\le4\eta^2(\alpha^2\|\hat{\bsx}_k-\bsx_k\|^2_{\bsL^2}+\alpha^2\|\bsx_k\|^2_{\bsL^2}+\|\beta\bm{v}_k+\bsg_k^0\|^2+\|\bsg_k-\bsg_k^0\|^2)\nonumber\\
		&\le4\eta^2\Big(\alpha^2\rho^2(L)\|\hat{\bsx}_k-\bsx_k\|^2+\|\bsx_k\|^2_{(\alpha^2\rho^2(L)+L_f^2)\bsK}+\Big\|\bm{v}_k+\frac{1}{\beta}\bsg_k^0\Big\|^2_{\beta^2\rho^2(L)\bsP}\Big),
	\end{align}
	where the first equality holds due to \eqref{nonconvex:kia-algo-dc-compact-x}; the first inequality holds due to \eqref{nonconvex:lemma:inequality-arithmetic-equ}; and the second inequality holds due to \eqref{nonconvex:KL-L-eq2}, \eqref{nonconvex:lemma-eq5}, and \eqref{nonconvex:gg1}.

	From \eqref{nonconvex:xminush_compress}, \eqref{nonconvex:xkoneminusx}, and \eqref{nonconvex:xminusxhat}, we have
	\begin{align}
		\mathbf{E}_{\mathcal{C}}[\|\bsx_{k+1}-\bsa_{k+1}\|^2]
		&\le\mathbf{E}_{\mathcal{C}}\Big[(1-c_4)\|\bsx_{k}-\bsa_{k}\|^2
		+\|\bsx_k\|^2_{4\eta^2(1+c_3^{-1})(\alpha^2\rho^2(L)+L_f^2)\bsK}\nonumber\\
		&\quad+\Big\|\bm{v}_k+\frac{1}{\beta}\bsg_k^0\Big\|^2_{4\eta^2(1+c_3^{-1})\beta^2\rho^2(L)\bsP}\Big].
		\label{nonconvex:xminush}
	\end{align}
	
	\noindent {\bf (vii)}  This step is to show the relation between $V_{k+1}$ and $V_{k}$.		
	We have
\begin{align}
	\mathbf{E}_{\mathcal{C}}[V_{k+1}]
	&\le  \mathbf{E}_{\mathcal{C}}\Big[V_{k}-\|\bsx_k\|^2_{\frac{\eta\alpha}{2}\bsL-\frac{\eta}{2}\bsK
		-\frac{\eta}{2}(1+3\eta)L_f^2\bsK}+\|\hat{\bm{x}}_k\|^2_{\frac{3\eta^2\alpha^2}{2}\bsL^2}\nonumber\\
	&\quad+\Big\|\bm{v}_k+\frac{1}{\beta}\bsg_k^0\Big\|^2_{\frac{6\eta^2\beta^2\rho(L)+\eta\beta}{4}\bsP}\nonumber\\
	&\quad+\frac{\eta}{2}(\alpha+2\beta)\rho(L)r_0\|\bsx_k-\bsa_{k}\|^2\nonumber\\
	&\quad+\|\hat{\bsx}\|^2_{\frac{\eta^2\beta}{2}(\alpha+\beta)\bsL+\frac{\eta^2}{2}\bsK}
	+\Big\|\bm{v}_k+\frac{1}{\beta}\bsg_{k}^0\Big\|^2_{\frac{\eta\beta}{4}\bsP}\nonumber\\
	&\quad+\eta^2 c_1L_f^2\|\bar{\bsg}_k\|^2+\|\hat{\bm{x}}_k\|^2_{\eta\beta\bsK
		+\eta^2(\beta^2\bsK-\alpha\beta\bsL)}\nonumber\\
	&\quad+\|\bm{x}_k\|^2_{\frac{\eta(\beta+2)}{4}\bsK+\frac{\eta}{4}(1+6\eta)L_f^2\bsK}
	+\Big(\eta^2c_2L_f^2+\frac{\eta}{8}\Big)\|\bar{\bsg}_k\|^2\nonumber\\
	&\quad-\Big\|\bm{v}_k+\frac{1}{\beta}\bsg_{k}^0\Big\|^2_{\eta(\beta-3\rho^{-1}_2(L))\bsP
		-\eta^2(\rho^{-1}_2(L)+\frac{\beta^2}{2}\rho(L))\bsP}\nonumber\\
	&\quad-\frac{\eta}{4}(1-2\eta L_f)\|\bar{\bsg}_{k}\|^2
	+\|\bsx_k\|^2_{\frac{\eta}{2}L_f^2\bsK}-\frac{\eta}{4}\|\bar{\bsg}_{k}^0\|^2\nonumber\\
	&\quad-c_4\|\bsx_{k}-\bsa_{k}\|^2+\|\bsx_k\|^2_{4\eta^2(1+c_3^{-1})(\alpha^2\rho^2(L)+L_f^2)\bsK}\nonumber\\
	&\quad+\Big\|\bm{v}_k+\frac{1}{\beta}\bsg_k^0\Big\|^2_{4\eta^2(1+c_3^{-1})\beta^2\rho^2(L)\bsP}\Big]\nonumber\\
	&=\mathbf{E}_{\mathcal{C}}\Big[V_{k}-\|\bsx_k\|^2_{\eta\bsM_1-\eta^2c_6\bsK}+\|\hat{\bm{x}}_k\|^2_{\eta\beta\bsK
		+\eta^2\bsM_2}\nonumber\\
	&\quad-\Big\|\bm{v}_k+\frac{1}{\beta}\bsg_k^0\Big\|^2_{\eta(\epsilon_3-\eta\epsilon_4)\bsP}
	-\eta(\epsilon_5-\eta\epsilon_6)\|\bar{\bsg}_{k}\|^2\nonumber\\
	&\quad-\Big(c_4-\frac{\eta}{2}(\alpha+2\beta)\rho(L)r_0\Big)\|\bsx_{k}-\bsa_{k}\|^2
	-\frac{\eta}{4}\|\bar{\bsg}_{k}^0\|^2\Big]\nonumber\\
	&\le  \mathbf{E}_{\mathcal{C}}\Big[V_{k}-\|\bsx_k\|^2_{\eta(c_5-\eta c_6)\bsK}+\|\hat{\bm{x}}_k\|^2_{\eta(\beta+\eta c_7)\bsK}\nonumber\\
	&\quad-\Big\|\bm{v}_k+\frac{1}{\beta}\bsg_k^0\Big\|^2_{\eta(\epsilon_3-\eta\epsilon_4)\bsP}
	-\eta(\epsilon_5-\eta\epsilon_6)\|\bar{\bsg}_{k}\|^2\nonumber\\
	&\quad-\Big(c_4-\frac{\eta}{2}(\alpha+2\beta)\rho(L)r_0\Big)\|\bsx_{k}-\bsa_{k}\|^2
	-\frac{\eta}{4}\|\bar{\bsg}_{k}^0\|^2\Big],\label{nonconvex:vkLya_xhat}
\end{align}
	where the first inequality holds due to \eqref{nonconvex:v1k}--\eqref{nonconvex:v4k-1}, \eqref{nonconvex:xminusxhat}, and \eqref{nonconvex:xminush}; and the second inequality holds due to \eqref{nonconvex:KL-L-eq2} and $\beta\le\alpha$.
	
	For the third term in the right-hand side of \eqref{nonconvex:vkLya_xhat}, from \eqref{nonconvex:lemma:inequality-arithmetic-equ} and $\rho(\bsK)=1$, we have
	\begin{align}\label{nonconvex:xhat}
		\|\hat{\bm{x}}_k\|^2_{\bsK}
		&=\|\hat{\bm{x}}_k-\bsx_k+\bsx_k\|^2_{\bsK}\le2\|\hat{\bm{x}}_k-\bsx_k\|^2+2\|\bsx_k\|^2_{\bsK}.
	\end{align}
	From \eqref{nonconvex:xminusxhat}, \eqref{nonconvex:vkLya_xhat} and \eqref{nonconvex:xhat}, we know that \eqref{nonconvex:vkLya-sm2} holds.
\end{proof}

We are now ready to prove Theorem~\ref{nonconvex:thm-sm}.

\noindent {\bf (i)}  We first show that all of the used constants are positive.

From $\alpha=\kappa_1\beta$, $\kappa_1\ge\frac{9+\kappa_4}{2\rho_2(L)}$, $\kappa_4>0$, and $\beta>\kappa_2\ge\frac{4+5L_f^2}{\kappa_4}$, we have
\begin{align}\label{nonconvex:beta1}
	\epsilon_1&=\frac{\kappa_1\beta}{2}\rho_2(L)-\frac{1}{4}(9\beta+4+5L_f^2)>\frac{\kappa_1\beta}{2}\rho_2(L)-\frac{1}{4}(9\beta+\kappa_4\beta)\ge0.
\end{align}
From $\beta>\kappa_2\ge\frac{6}{\rho_2(L)}$, we have
\begin{align}\label{nonconvex:beta2}
	\epsilon_3>0.
\end{align}
From $\alpha=\kappa_1\beta$ and $\beta>\kappa_2=\max\{\kappa_5,~\sqrt{\kappa_6}\}$, we have
\begin{align}\label{nonconvex:beta3}
	\epsilon_5&=\frac{1}{8}-\frac{(\kappa_1+1)^2L_f^2}{\beta^3\rho_2(L)}
	-\frac{L_f^2}{2\beta^2\rho^2_2(L)}>\frac{1}{8}-\frac{(\kappa_1+1)^2L_f^2}{\kappa_5\beta^2\rho_2(L)}
	-\frac{L_f^2}{2\beta^2\rho^2_2(L)}>0.
\end{align}

From \eqref{nonconvex:beta1}--\eqref{nonconvex:beta3},  and $0<\eta<\kappa_3$, we have
\begin{subequations}
	\begin{align}
		&\eta(\epsilon_1-\eta\epsilon_2)>0,\label{nonconvex:vkLya1.1}\\
		&\eta(\epsilon_3-\eta\epsilon_4)>0,\label{nonconvex:vkLya1.2}\\
		&\eta(\epsilon_5-\eta\epsilon_6)>0,\label{nonconvex:vkLya1}\\
		&\epsilon_7-\eta\epsilon_8-\eta^2\epsilon_9>0.\label{nonconvex:vkLya1compress}
	\end{align}
\end{subequations}

\noindent {\bf (ii)}  We then show that \eqref{nonconvex:thm-sm-equ1} and \eqref{nonconvex:thm-sm-equ2} hold.

From \eqref{nonconvex:lemma:inequality-arithmetic-equ}, we have
\begin{subequations}
	\begin{align}
		V_{k}&\ge\frac{1}{2}\|\bsx_{k}\|^2_{\bsK}
		+\frac{1}{2}\Big(1+\frac{\alpha}{\beta}\Big)\Big\|\bsv_k+\frac{1}{\beta}\bsg_k^0\Big\|^2_{\bsP}\nonumber\\
		&\quad-\frac{\beta}{2\alpha\rho_2(L)}\|\bsx_{k}\|^2_{\bsK}
		-\frac{\alpha}{2\beta}\Big\|\bsv_k+\frac{1}{\beta}\bsg_k^0\Big\|^2_{\bsP}+n(f(\bar{x}_k)-f^*)+\|\bsx_{k}-\bsa_{k}\|^2\nonumber\\
		&\ge\epsilon_{10}\Big(\|\bsx_{k}\|^2_{\bsK}
		+\Big\|\bsv_k+\frac{1}{\beta}\bsg_k^0\Big\|^2_{\bsP}\Big)+n(f(\bar{x}_k)-f^*)+\|\bsx_{k}-\bsa_{k}\|^2\label{nonconvex:vkLya3.2}\\
		&\ge\epsilon_{10}\hat{V}_k\ge0.\label{nonconvex:vkLya3}
	\end{align}
\end{subequations}

From \eqref{nonconvex:vkLya-sm2} and \eqref{nonconvex:vkLya1.2}--\eqref{nonconvex:vkLya1compress}, we have
\begin{align}\label{nonconvex:vkLya-sm4}
	&\mathbf{E}_{\mathcal{C}}[V_{T+1}]\le  V_{0}-\sum_{k=0}^{T}\mathbf{E}_{\mathcal{C}}[\|\bsx_k\|^2_{\eta(\epsilon_1-\eta\epsilon_2)\bsK}]
	-\sum_{k=0}^{T}\mathbf{E}_{\mathcal{C}}\Big[\frac{\eta}{4}\|\bar{\bsg}_{k}^0\|^2\Big].
\end{align}

From \eqref{nonconvex:vkLya-sm4}, \eqref{nonconvex:vkLya1.1}, and \eqref{nonconvex:vkLya3}, we have
\begin{align*}
	\sum_{k=0}^{T}\mathbf{E}_{\mathcal{C}}[\|\bsx_k\|^2_{\bsK}+\|\bar{\bsg}_{k}^0\|^2]
	\le\frac{V_0}{\min\{\eta(\epsilon_1-\eta\epsilon_2),~\frac{\eta}{4}\}},
\end{align*}
which yields \eqref{nonconvex:thm-sm-equ1}.

From \eqref{nonconvex:vkLya-sm4}, \eqref{nonconvex:vkLya1.1}, and \eqref{nonconvex:vkLya3.2}, we have
\begin{align*}
	\mathbf{E}_{\mathcal{C}}[n(f(\bar{x}_T)-f^*)]\le \mathbf{E}_{\mathcal{C}}[V_T]\le V_0,
\end{align*}
which yields \eqref{nonconvex:thm-sm-equ2}.

\subsection{Proof of Theorem~\ref{nonconvex:thm-ft}}\label{nonconvex:proof-thm-ft}
In this proof, in addition to the notations used in the proof of Theorem~\ref{nonconvex:thm-sm}, we also denote
\begin{align*}
	\epsilon&=\frac{\epsilon_{12}}{\epsilon_{11}},~
	\epsilon_{11}=\max\Big\{\frac{1}{2}+\frac{\alpha}{\beta},
	~\frac{\alpha\rho_2(L)+\beta}{2\alpha\rho_2(L)}\Big\},\\
	\epsilon_{12}&=\eta\min\Big\{\epsilon_1-\eta\epsilon_2,~\epsilon_3-\eta\epsilon_4,
	~\frac{\nu}{2},~\frac{\epsilon_7}{\eta}-\epsilon_8-\eta\epsilon_9\Big\}.
\end{align*}

\noindent {\bf (i)}  We first show that $\epsilon\in(0,1)$.

From \eqref{nonconvex:vkLya1.1}--\eqref{nonconvex:vkLya1compress}, we have
\begin{align}\label{nonconvex:vkLya1.4}
	\epsilon_{12}>0~\text{and}~\epsilon=\frac{\epsilon_{12}}{\epsilon_{11}}>0.
\end{align}

Noting that $\epsilon_{11}\ge\frac{1}{2}+\frac{\alpha}{\beta}\ge\frac{3}{2}$, and $\epsilon_{12}<\epsilon_{7}=\frac{\varphi\psi r}{2}(1+\varphi\psi r)<1$ due to $\varphi\psi r\in(0,1)$, we have
\begin{align}\label{nonconvex:vkLya-epsilon}
	0<\epsilon=\frac{\epsilon_{12}}{\epsilon_{11}}<1.
\end{align}

\noindent {\bf (ii)}  We then show that \eqref{nonconvex:thm-ft-equ1} holds.

From \eqref{nonconvex:lemma:inequality-arithmetic-equ}, we have
\begin{align}\label{nonconvex:vkLya3.1}
	V_k\le\epsilon_{11}\hat{V}_k.
\end{align}
From Assumptions~\ref{nonconvex:ass:optset} and \ref{nonconvex:ass:fil} as well as \eqref{nonconvex:equ:plc}, we have that
\begin{align}\label{nonconvex:gg3}
	\|\bar{\bsg}^0_k\|^2=n\|\nabla f(\bar{x}_k)\|^2\ge2\nu n(f(\bar{x}_k)-f^*).
\end{align}
Then, from \eqref{nonconvex:vkLya-sm2},  \eqref{nonconvex:vkLya1}, \eqref{nonconvex:vkLya1.4}, and \eqref{nonconvex:vkLya3.1}--\eqref{nonconvex:gg3}, we have
\begin{align}\label{nonconvex:vkLya2.1}
	\mathbf{E}_{\mathcal{C}}[V_{k+1}]
	\le \mathbf{E}_{\mathcal{C}}[V_{k}-\epsilon_{12}\hat{V}_k]
	\le \mathbf{E}_{\mathcal{C}}\Big[V_{k}-\frac{\epsilon_{12}}{\epsilon_{11}}V_{k}\Big].
\end{align}
Hence, from \eqref{nonconvex:vkLya2.1} and  \eqref{nonconvex:vkLya-epsilon}, we have
\begin{align*}
	\mathbf{E}_{\mathcal{C}}[V_{k+1}]&\le  (1-\epsilon)\mathbf{E}_{\mathcal{C}}[V_k]\le (1-\epsilon)^{k+1}V_0,
\end{align*}
which yields \eqref{nonconvex:thm-ft-equ1}.

\subsection{Proof of Theorem~\ref{nonconvex:thm-sm_ef}}\label{nonconvex:proof-thm-sm_ef}
In this proof, in addition to the notations used in the proof of Theorem~\ref{nonconvex:thm-sm}, we also denote
\begin{align*}
\kappa_0&=\frac{1}{\sqrt{r_0r_1}},~
\check{\kappa}_3=\min\Big\{\frac{\epsilon_1}{\epsilon_2},~\frac{\epsilon_3}{\epsilon_4},
~\frac{\epsilon_5}{\epsilon_6},~\tau_1,~\tau_2\Big\},\\
\check{\epsilon}_8&=((\alpha+2\beta)\rho(L)+4\beta)(r_0r_1+1)\sigma^2,\\
\check{\epsilon}_9&=2(7\alpha^2\rho^2(L)+2\beta^2+1)(r_0r_1+1)\sigma^2,\\
\epsilon_{10}&=\frac{\alpha\rho_2(L)-\beta}{2\alpha\rho_2(L)},\\
\check{c}_4&=c_3+2c_3^2-16\eta^2(1+c_3^{-1})\alpha^2\rho^2(L)r_0r_2,\\
r_1&>1,~r_2=\frac{r_1}{r_1-1},~\tau_0\in(0,\frac{\epsilon_7}{r_0r_2}),\\
\tau_1&=\frac{\sqrt{r_2^2\epsilon_8^2+2r_2\epsilon_9(\epsilon_7-r_0r_2\tau_0)}
	-r_2\epsilon_8}{2r_2\epsilon_9},\\
\tau_2&=\frac{\sqrt{\check{\epsilon}_8^2+4(1-r_0r_1\sigma^2)\tau_0\check{\epsilon}_9}
	-\check{\epsilon}_8}{2\check{\epsilon}_9},\\
\tau_3&=\epsilon_7-r_0r_2\tau_0-2\eta r_2\epsilon_8-2\eta^2 r_2\epsilon_9,\\
\tau_4&=(1-r_0r_1\sigma^2)\tau_0-\eta\check{\epsilon}_8
-\eta^2\check{\epsilon}_9,\\
W_{k}&=V_{k}+\tau_0\|\bse_{k}\|^2,~\hat{W}_k=\hat{V}_k+\tau_0\|\bse_{k}\|^2.
\end{align*}

\noindent {\bf (i)}  We first show that all of the used constants are positive.

Noting that the settings on $\alpha$ and $\beta$ in both Theorems~\ref{nonconvex:thm-sm} and  \ref{nonconvex:thm-sm_ef} are the same, \eqref{nonconvex:beta1}--\eqref{nonconvex:beta3} still hold. From \eqref{nonconvex:beta1}--\eqref{nonconvex:beta3} and $\eta\in(0,\check{\kappa}_3)$, we know that \eqref{nonconvex:vkLya1.1}--\eqref{nonconvex:vkLya1} still hold.

From $\tau_0\in(0,\frac{\epsilon_7}{r_0r_2})$, we have
\begin{align}\label{nonconvex:epsilon_7}
\epsilon_7-r_0r_2\tau_0>0~\text{and}~\tau_1>0.
\end{align}
Then, from \eqref{nonconvex:epsilon_7} and $\eta\in(0,\tau_1)$, we have
\begin{align}\label{nonconvex:vkLya1compress_ef}
\tau_3=\epsilon_7-r_0r_2\tau_0-2\eta r_2\epsilon_8-2\eta^2 r_2\epsilon_9>0.
\end{align}

From $\sigma\in(0,\kappa_0)$ and $\tau_0>0$ we have
\begin{align}\label{nonconvex:eta_0}
(1-r_0r_1\sigma^2)\tau_0>0~\text{and}~\tau_2>0.
\end{align}
Then, from \eqref{nonconvex:eta_0} and $\eta\in(0,\tau_2)$, we have
\begin{align}\label{nonconvex:vkLya1compress_ef2}
\tau_4=(1-r_0r_1\sigma^2)\tau_0-\eta\check{\epsilon}_8-\eta^2\check{\epsilon}_9>0.
\end{align}

\noindent {\bf (ii)}  We then show that \eqref{nonconvex:thm-sm_ef-equ1} and \eqref{nonconvex:thm-sm_ef-equ2} hold.

Noting that the compact form of \eqref{nonconvex:kia-algo-dc_ef-x} and \eqref{nonconvex:kia-algo-dc_ef-v}  respectively can be rewritten as \eqref{nonconvex:kia-algo-dc-compact-x} and \eqref{nonconvex:kia-algo-dc-compact-v}, we know that \eqref{nonconvex:v1k}--\eqref{nonconvex:v4k-1}, and \eqref{nonconvex:xkoneminusx} still hold. Moreover, \eqref{nonconvex:xminush_compress} still holds since the  compact form of \eqref{nonconvex:kia-algo-dc_ef-a} and \eqref{nonconvex:kia-algo-dc_ef-q} is \eqref{nonconvex:kia-algo-dc-compact-a} and \eqref{nonconvex:kia-algo-dc-compact-q}, respectively.

From \eqref{nonconvex:kia-algo-dc_ef-qhat}, \eqref{nonconvex:kia-algo-dc_ef-e}, and \eqref{nonconvex:ass:compression_equ}, we have
\begin{align}\label{nonconvex:ek}
\mathbf{E}_{\mathcal{C}}[\|\bse_{k+1}\|^2]
&\le r_0\|\sigma\bse_k+\bsx_k-\bsa_{k}\|^2\le r_0r_1\sigma^2\|\bse_k\|^2+r_0r_2\|\bsx_k-\bsa_{k}\|^2.
\end{align}

We have
\begin{align}\label{nonconvex:xminusxhat_ef}
\mathbf{E}_{\mathcal{C}}[\|\bsx_k-\hat{\bsx}_k\|^2]
&=\mathbf{E}_{\mathcal{C}}[\|\bse_{k+1}-\sigma\bse_{k}\|^2]\nonumber\\
&\le \mathbf{E}_{\mathcal{C}}[2\|\bse_{k+1}\|^2+2\sigma^2\|\bse_{k}\|^2]\nonumber\\
&\le\mathbf{E}_{\mathcal{C}}[2(r_0r_1+1)\sigma^2\|\bse_{k}\|^2+2r_0r_2\|\bsx_k-\bsa_{k}\|^2],
\end{align}
where the first equality holds due to \eqref{nonconvex:kia-algo-dc_ef-compact-xhat} and \eqref{nonconvex:kia-algo-dc_ef-e}; and the last inequality holds due to \eqref{nonconvex:ek}.

From \eqref{nonconvex:xminush_compress}, \eqref{nonconvex:xkoneminusx}, \eqref{nonconvex:xminusxhat_ef}, we have
\begin{align}
\mathbf{E}_{\mathcal{C}}[\|\bsx_{k+1}-\bsa_{k+1}\|^2]
&\le\mathbf{E}_{\mathcal{C}}\Big[(1-\check{c}_4)\|\bsx_{k}-\bsa_{k}\|^2
+\|\bsx_k\|^2_{4\eta^2(1+c_3^{-1})(\alpha^2\rho^2(L)+L_f^2)\bsK}\nonumber\\
&\quad+\Big\|\bm{v}_k+\frac{1}{\beta}\bsg_k^0\Big\|^2_{4\eta^2(1+c_3^{-1})\beta^2\rho^2(L)\bsP}
+8\eta^2\alpha^2\rho^2(L)(r_0r_1+1)\sigma^2\|\bse_{k}\|^2\Big].
\label{nonconvex:xminush_ef}
\end{align}

Similar to the way to get \eqref{nonconvex:vkLya_xhat}, from \eqref{nonconvex:v1k}--\eqref{nonconvex:v4k-1}, \eqref{nonconvex:xminusxhat_ef}, and \eqref{nonconvex:xminush_ef}, we have
\begin{align}
\mathbf{E}_{\mathcal{C}}[V_{k+1}]
&\le  \mathbf{E}_{\mathcal{C}}\Big[V_{k}-\|\bsx_k\|^2_{\frac{\eta\alpha}{2}\bsL-\frac{\eta}{2}\bsK
-\frac{\eta}{2}(1+3\eta)L_f^2\bsK}+\|\hat{\bm{x}}_k\|^2_{\frac{3\eta^2\alpha^2}{2}\bsL^2}\nonumber\\
&\quad+\Big\|\bm{v}_k+\frac{1}{\beta}\bsg_k^0\Big\|^2_{\frac{6\eta^2\beta^2\rho(L)+\eta\beta}{4}\bsP}\nonumber\\
&\quad+\eta(\alpha+2\beta)\rho(L)((r_0r_1+1)\sigma^2\|\bse_{k}\|^2+r_0r_2\|\bsx_k-\bsa_{k}\|^2)\nonumber\\
&\quad+\|\hat{\bsx}\|^2_{\frac{\eta^2\beta}{2}(\alpha+\beta)\bsL+\frac{\eta^2}{2}\bsK}
+\Big\|\bm{v}_k+\frac{1}{\beta}\bsg_{k}^0\Big\|^2_{\frac{\eta\beta}{4}\bsP}\nonumber\\
&\quad+\eta^2 c_1L_f^2\|\bar{\bsg}_k\|^2+\|\hat{\bm{x}}_k\|^2_{\eta\beta\bsK
+\eta^2(\beta^2\bsK-\alpha\beta\bsL)}\nonumber\\
&\quad+\|\bm{x}_k\|^2_{\frac{\eta(\beta+2)}{4}\bsK+\frac{\eta}{4}(1+6\eta)L_f^2\bsK}
+\Big(\eta^2c_2L_f^2+\frac{\eta}{8}\Big)\|\bar{\bsg}_k\|^2\nonumber\\
&\quad-\Big\|\bm{v}_k+\frac{1}{\beta}\bsg_{k}^0\Big\|^2_{\eta(\beta-3\rho^{-1}_2(L))\bsP
-\eta^2(\rho^{-1}_2(L)+\frac{\beta^2}{2}\rho(L))\bsP}\nonumber\\
&\quad-\frac{\eta}{4}(1-2\eta L_f)\|\bar{\bsg}_{k}\|^2
+\|\bsx_k\|^2_{\frac{\eta}{2}L_f^2\bsK}-\frac{\eta}{4}\|\bar{\bsg}_{k}^0\|^2\nonumber\\
&\quad-\check{c}_4\|\bsx_{k}-\bsa_{k}\|^2+\|\bsx_k\|^2_{4\eta^2(1+c_3^{-1})(\alpha^2\rho^2(L)+L_f^2)\bsK}\nonumber\\
&\quad+\Big\|\bm{v}_k+\frac{1}{\beta}\bsg_k^0\Big\|^2_{4\eta^2(1+c_3^{-1})\beta^2\rho^2(L)\bsP}\nonumber\\
&\quad+8\eta^2\alpha^2\rho^2(L)(r_0r_1+1)\sigma^2\|\bse_{k}\|^2\Big]\nonumber\\
&=\mathbf{E}_{\mathcal{C}}\Big[V_{k}-\|\bsx_k\|^2_{\eta\bsM_1-\eta^2c_6\bsK}+\|\hat{\bm{x}}_k\|^2_{\eta\beta\bsK
+\eta^2\bsM_2}\nonumber\\
&\quad-\Big\|\bm{v}_k+\frac{1}{\beta}\bsg_k^0\Big\|^2_{\eta(\epsilon_3-\eta\epsilon_4)\bsP}
-\eta(\epsilon_5-\eta\epsilon_6)\|\bar{\bsg}_{k}\|^2\nonumber\\
&\quad-\Big(\check{c}_4-2\eta(\alpha+2\beta)\rho(L)r_0r_2\Big)\|\bsx_{k}-\bsa_{k}\|^2
-\frac{\eta}{4}\|\bar{\bsg}_{k}^0\|^2\nonumber\\
&\quad+\eta(\alpha+2\beta+8\alpha^2\rho(L)\eta)\rho(L)(r_0r_1+1)\sigma^2\|\bse_{k}\|^2\Big]\nonumber\\
&\le  \mathbf{E}_{\mathcal{C}}\Big[V_{k}-\|\bsx_k\|^2_{\eta(c_5-\eta c_6)\bsK}+\|\hat{\bm{x}}_k\|^2_{\eta(\beta+\eta c_7)\bsK}\nonumber\\
&\quad-\Big\|\bm{v}_k+\frac{1}{\beta}\bsg_k^0\Big\|^2_{\eta(\epsilon_3-\eta\epsilon_4)\bsP}
-\eta(\epsilon_5-\eta\epsilon_6)\|\bar{\bsg}_{k}\|^2\nonumber\\
&\quad-(\check{c}_4-2\eta(\alpha+2\beta)\rho(L)r_0r_2)\|\bsx_{k}-\bsa_{k}\|^2
-\frac{\eta}{4}\|\bar{\bsg}_{k}^0\|^2\nonumber\\
&\quad+\eta(\alpha+2\beta+8\alpha^2\rho(L)\eta)\rho(L)(r_0r_1+1)\sigma^2\|\bse_{k}\|^2\Big].
\label{nonconvex:vkLya_xhat_ef}
\end{align}

From \eqref{nonconvex:ek}, \eqref{nonconvex:xminusxhat_ef}, \eqref{nonconvex:vkLya_xhat_ef} and \eqref{nonconvex:xhat}, we have
\begin{align}\label{nonconvex:vkLya-sm2_ef}
\mathbf{E}_{\mathcal{C}}[W_{k+1}]
&\le  \mathbf{E}_{\mathcal{C}}\Big[W_{k}-\frac{\eta}{4}\|\bar{\bsg}_{k}^0\|^2
-\|\bsx_k\|^2_{\eta(\epsilon_1-\eta \epsilon_2)\bsK}\nonumber\\
&\quad-\Big\|\bm{v}_k+\frac{1}{\beta}\bsg_k^0\Big\|^2_{\eta(\epsilon_3-\eta\epsilon_4)\bsP}
-\eta(\epsilon_5-\eta\epsilon_6)\|\bar{\bsg}_{k}\|^2-\tau_3\|\bsx_{k}-\bsa_{k}\|^2-\tau_4\|\bse_k\|^2\Big].
\end{align}

Same as the way to get \eqref{nonconvex:vkLya3.2} and \eqref{nonconvex:vkLya3}, we have
\begin{subequations}
\begin{align}
W_{k}
&\ge\epsilon_{10}\Big(\|\bsx_{k}\|^2_{\bsK}
+\Big\|\bsv_k+\frac{1}{\beta}\bsg_k^0\Big\|^2_{\bsP}\Big)
+n(f(\bar{x}_k)-f^*)+\|\bsx_{k}-\bsa_{k}\|^2+\tau_0\|\bse_{k}\|^2\label{nonconvex:vkLya3.2_ef}\\
&\ge\epsilon_{10}\hat{W}_k\ge0.\label{nonconvex:vkLya3_ef}
\end{align}
\end{subequations}

From \eqref{nonconvex:vkLya-sm2_ef}, \eqref{nonconvex:vkLya1.2}--\eqref{nonconvex:vkLya1} \eqref{nonconvex:vkLya1compress_ef}, and \eqref{nonconvex:vkLya1compress_ef2}, we have
\begin{align}\label{nonconvex:vkLya-sm4_ef}
&\mathbf{E}_{\mathcal{C}}[W_{T+1}]\le  W_{0}-\sum_{k=0}^{T}\mathbf{E}_{\mathcal{C}}[\|\bsx_k\|^2_{\eta(\epsilon_1-\eta\epsilon_2)\bsK}]
-\sum_{k=0}^{T}\mathbf{E}_{\mathcal{C}}\Big[\frac{\eta}{4}\|\bar{\bsg}_{k}^0\|^2\Big].
\end{align}

From \eqref{nonconvex:vkLya-sm4_ef}, \eqref{nonconvex:vkLya1.1}, and \eqref{nonconvex:vkLya3_ef}, we have
\begin{align*}
\sum_{k=0}^{T}\mathbf{E}_{\mathcal{C}}[\|\bsx_k\|^2_{\bsK}+\|\bar{\bsg}_{k}^0\|^2]
\le\frac{W_0}{\min\{\eta(\epsilon_1-\eta\epsilon_2),~\frac{\eta}{4}\}},
\end{align*}
which yields \eqref{nonconvex:thm-sm_ef-equ1}.

From \eqref{nonconvex:vkLya-sm4_ef}, \eqref{nonconvex:vkLya1.1}, and \eqref{nonconvex:vkLya3.2_ef}, we have
\begin{align*}
\mathbf{E}_{\mathcal{C}}[n(f(\bar{x}_T)-f^*)]\le \mathbf{E}_{\mathcal{C}}[W_T]\le W_0,
\end{align*}
which yields \eqref{nonconvex:thm-sm_ef-equ2}.

\subsection{Proof of Theorem~\ref{nonconvex:thm-ft_ef}}\label{nonconvex:proof-thm-ft_ef}
In this proof, in addition to the notations used in the proofs of Theorems~\ref{nonconvex:thm-sm}--\ref{nonconvex:thm-sm_ef}, we also denote
\begin{align*}
\check{\epsilon}&=\frac{\check{\epsilon}_{12}}{\epsilon_{11}},~
\check{\epsilon}_{12}=\eta\min\Big\{\epsilon_1-\eta\epsilon_2,~\epsilon_3-\eta\epsilon_4,
~\frac{\nu}{2},~\frac{\tau_3}{\eta},~\frac{\tau_4}{\eta\tau_0}\Big\}.
\end{align*}

\noindent {\bf (i)}  We first show that $\check{\epsilon}\in(0,1)$.

From \eqref{nonconvex:vkLya1.2}--\eqref{nonconvex:vkLya1} \eqref{nonconvex:vkLya1compress_ef}, and \eqref{nonconvex:vkLya1compress_ef2}, we have
\begin{align}\label{nonconvex:vkLya1.4_ef}
\check{\epsilon}_{12}>0~\text{and}~\check{\epsilon}=\frac{\check{\epsilon}_{12}}{\epsilon_{11}}>0.
\end{align}

From $\tau_2<\epsilon_7$ and \eqref{nonconvex:vkLya-epsilon}, we have
\begin{align}\label{nonconvex:vkLya-epsilon_ef}
0<\check{\epsilon}=\frac{\check{\epsilon}_{12}}{\epsilon_{11}}\le\frac{\epsilon_{12}}{\epsilon_{11}}<1.
\end{align}

\noindent {\bf (ii)}  We then show that \eqref{nonconvex:thm-ft_ef-equ1} holds.

From \eqref{nonconvex:lemma:inequality-arithmetic-equ}, we have
\begin{align}\label{nonconvex:vkLya3.1_ef}
W_k\le\epsilon_{11}\hat{W}_k.
\end{align}

Then, from \eqref{nonconvex:vkLya-sm2_ef},  \eqref{nonconvex:vkLya1}, \eqref{nonconvex:vkLya1.4_ef}, \eqref{nonconvex:vkLya3.1_ef}, and\eqref{nonconvex:gg3}, we have
\begin{align}\label{nonconvex:vkLya2.1_ef}
\mathbf{E}_{\mathcal{C}}[W_{k+1}]
\le \mathbf{E}_{\mathcal{C}}[W_{k}-\check{\epsilon}_{12}\hat{W}_k]
\le \mathbf{E}_{\mathcal{C}}\Big[W_{k}-\frac{\check{\epsilon}_{12}}{\epsilon_{11}}W_{k}\Big].
\end{align}
Hence, from \eqref{nonconvex:vkLya2.1_ef} and  \eqref{nonconvex:vkLya-epsilon_ef}, we have
\begin{align*}
\mathbf{E}_{\mathcal{C}}[W_{k+1}]&\le  (1-\check{\epsilon})\mathbf{E}_{\mathcal{C}}[W_k]\le (1-\check{\epsilon})^{k+1}W_0,
\end{align*}
which yields \eqref{nonconvex:thm-ft_ef-equ1}.

\subsection{Proof of Theorem~\ref{nonconvex:thm-sm_quantization}}\label{nonconvex:proof-thm-sm_quantization}
In this proof, in addition to the notations used in the proof of Theorem~\ref{nonconvex:thm-sm}, we also denote
\begin{align*}
\tilde{\kappa}_3&=\min\Big\{\frac{\epsilon_1}{\tilde{\epsilon}_2},~\frac{\epsilon_3}{\tilde{\epsilon}_4},
~\frac{\epsilon_5}{\epsilon_6}\Big\},\\
\tilde{\epsilon}_2&=3L_f^2+2\beta^2+1+3\alpha^2\rho^2(L),\\
\tilde{\epsilon}_4&=2\beta^2\rho(L)+\rho_2^{-1}(L),\\
\tilde{\epsilon}_8&=\frac{1}{2}(\alpha+2\beta)\rho(L)+2\beta,\\
c_{8}&=\eta(\tilde{\epsilon}_8+2c_7\eta)n\tilde{d}^2Cs_0^2.
\end{align*}

\noindent {\bf (i)}  We first show that all of the used constants are positive.

Noting that the settings on $\alpha$ and $\beta$ in both Theorems~\ref{nonconvex:thm-sm} and  \ref{nonconvex:thm-sm_quantization} are the same, \eqref{nonconvex:beta1}--\eqref{nonconvex:beta3} still hold.
From \eqref{nonconvex:beta1}--\eqref{nonconvex:beta3},  and $0<\eta<\tilde{\kappa}_3$, we have
\begin{subequations}
\begin{align}
&\eta(\epsilon_1-\eta\tilde{\epsilon}_2)>0,\label{nonconvex:vkLya1.1_quantization}\\
&\eta(\epsilon_3-\eta\tilde{\epsilon}_4)>0,\label{nonconvex:vkLya1.2_quantization}\\
&\eta(\epsilon_5-\eta\epsilon_6)>0.\label{nonconvex:vkLya1_quantization}
\end{align}
\end{subequations}

\noindent {\bf (ii)}  We then show that \eqref{nonconvex:thm-sm-equ1_quantization} and \eqref{nonconvex:thm-sm-equ2_quantization} hold.

Noting that \eqref{nonconvex:kia-algo-dc-x_determin} and \eqref{nonconvex:kia-algo-dc-v_determin} can respectively be rewritten as \eqref{nonconvex:kia-algo-dc-x-compress} and \eqref{nonconvex:kia-algo-dc-v-compress}, we know that \eqref{nonconvex:v1k}--\eqref{nonconvex:v4k-1}, and \eqref{nonconvex:xkoneminusx} still hold. 

Similar to the way to get \eqref{nonconvex:vkLya_xhat}, from \eqref{nonconvex:v1k}--\eqref{nonconvex:v4k-1}, we have
\begin{align}\label{nonconvex:vkLya_xhat_determin}
U_{k+1}
&\le  U_{k}-\|\bsx_k\|^2_{\frac{\eta\alpha}{2}\bsL-\frac{\eta}{2}\bsK
-\frac{\eta}{2}(1+3\eta)L_f^2\bsK}+\|\hat{\bm{x}}_k\|^2_{\frac{3\eta^2\alpha^2}{2}\bsL^2}\nonumber\\
&\quad+\Big\|\bm{v}_k+\frac{1}{\beta}\bsg_k^0\Big\|^2_{\frac{6\eta^2\beta^2\rho(L)+\eta\beta}{4}\bsP}\nonumber\\
&\quad+\frac{\eta}{2}(\alpha+2\beta)\rho(L)\|\bsx_k-\hat{\bsx}_{k}\|^2\nonumber\\
&\quad+\|\hat{\bsx}\|^2_{\frac{\eta^2\beta}{2}(\alpha+\beta)\bsL+\frac{\eta^2}{2}\bsK}
+\Big\|\bm{v}_k+\frac{1}{\beta}\bsg_{k}^0\Big\|^2_{\frac{\eta\beta}{4}\bsP}\nonumber\\
&\quad+\eta^2 c_1L_f^2\|\bar{\bsg}_k\|^2+\|\hat{\bm{x}}_k\|^2_{\eta\beta\bsK
+\eta^2(\beta^2\bsK-\alpha\beta\bsL)}\nonumber\\
&\quad+\|\bm{x}_k\|^2_{\frac{\eta(\beta+2)}{4}\bsK+\frac{\eta}{4}(1+6\eta)L_f^2\bsK}
+\Big(\eta^2c_2L_f^2+\frac{\eta}{8}\Big)\|\bar{\bsg}_k\|^2\nonumber\\
&\quad-\Big\|\bm{v}_k+\frac{1}{\beta}\bsg_{k}^0\Big\|^2_{\eta(\beta-3\rho^{-1}_2(L))\bsP
-\eta^2(\rho^{-1}_2(L)+\frac{\beta^2}{2}\rho(L))\bsP}\nonumber\\
&\quad-\frac{\eta}{4}(1-2\eta L_f)\|\bar{\bsg}_{k}\|^2
+\|\bsx_k\|^2_{\frac{\eta}{2}L_f^2\bsK}-\frac{\eta}{4}\|\bar{\bsg}_{k}^0\|^2\nonumber\\
&=U_{k}-\|\bsx_k\|^2_{\eta\bsM_1-3\eta^2L_f^2\bsK}+\|\hat{\bm{x}}_k\|^2_{\eta\beta\bsK
+\eta^2\bsM_2}\nonumber\\
&\quad-\Big\|\bm{v}_k+\frac{1}{\beta}\bsg_k^0\Big\|^2_{\eta(\epsilon_3-\eta\tilde{\epsilon}_4)\bsP}
-\eta(\epsilon_5-\eta\epsilon_6)\|\bar{\bsg}_{k}\|^2\nonumber\\
&\quad+\frac{\eta}{2}(\alpha+2\beta)\rho(L)\|\bsx_{k}-\hat{\bsx}_{k}\|^2
-\frac{\eta}{4}\|\bar{\bsg}_{k}^0\|^2\nonumber\\
&\le  U_{k}-\|\bsx_k\|^2_{\eta(c_5-3\eta L_f^2)\bsK}+\|\hat{\bm{x}}_k\|^2_{\eta(\beta+\eta c_7)\bsK}\nonumber\\
&\quad-\Big\|\bm{v}_k+\frac{1}{\beta}\bsg_k^0\Big\|^2_{\eta(\epsilon_3-\eta\tilde{\epsilon}_4)\bsP}
-\eta(\epsilon_5-\eta\epsilon_6)\|\bar{\bsg}_{k}\|^2\nonumber\\
&\quad+\frac{\eta}{2}(\alpha+2\beta)\rho(L)\|\bsx_{k}-\hat{\bsx}_{k}\|^2
-\frac{\eta}{4}\|\bar{\bsg}_{k}^0\|^2.
\end{align}

From \eqref{nonconvex:xhat} and \eqref{nonconvex:vkLya_xhat_determin}, we have
\begin{align}\label{nonconvex:vkLya-sm2_determin}
U_{k+1}
&\le  U_{k}-\|\bsx_k\|^2_{\eta(\epsilon_1-\eta\tilde{\epsilon}_2)\bsK}
-\Big\|\bm{v}_k+\frac{1}{\beta}\bsg_k^0\Big\|^2_{\eta(\epsilon_3-\eta\tilde{\epsilon}_4)\bsP}\nonumber\\
&\quad-\eta(\epsilon_5-\eta\epsilon_6)\|\bar{\bsg}_{k}\|^2-\frac{\eta}{4}\|\bar{\bsg}_{k}^0\|^2
+\eta(\tilde{\epsilon}_8+2c_7\eta)\|\bsx_{k}-\hat{\bsx}_{k}\|^2.
\end{align}

From Lemma~\ref{nonconvex:lemma:pnorm}, we have
\begin{align}\label{nonconvex:xminusxhat_determin}
\|\bsx_{k}-\hat{\bsx}_{k}\|^2&=\sum_{i=1}^{n}\|x_{i,k}-\hat{x}_{i,k}\|^2
\le\sum_{i=1}^{n}\tilde{d}^2\|x_{i,k}-\hat{x}_{i,k}\|^2_p\le n\tilde{d}^2\max_{i\in[n]}\|x_{i,k}-\hat{x}_{i,k}\|^2_p.
\end{align}

Same as the way to get \eqref{nonconvex:vkLya3.2} and \eqref{nonconvex:vkLya3}, we have
\begin{subequations}
\begin{align}
U_{k}
&\ge\epsilon_{10}\Big(\|\bsx_{k}\|^2_{\bsK}
+\Big\|\bsv_k+\frac{1}{\beta}\bsg_k^0\Big\|^2_{\bsP}\Big)
+n(f(\bar{x}_k)-f^*)\label{nonconvex:vkLya3.2_determin}\\
&\ge\epsilon_{10}\hat{U}_k\ge0.\label{nonconvex:vkLya3_determin}
\end{align}
\end{subequations}

From \eqref{nonconvex:vkLya-sm2_determin}--\eqref{nonconvex:xminusxhat_quantization2} and \eqref{nonconvex:vkLya1.2_quantization}--\eqref{nonconvex:vkLya1_quantization}, we have
\begin{align}\label{nonconvex:vkLya-sm4_quantization}
\mathbf{E}_{\mathcal{C}}[U_{T+1}]&\le  U_{0}-\sum_{k=0}^{T}\mathbf{E}_{\mathcal{C}}[\|\bsx_k\|^2_{\eta(\epsilon_1-\eta\tilde{\epsilon}_2)\bsK}
+\frac{\eta}{4}\|\bar{\bsg}_{k}^0\|^2]+\frac{c_{8}}{1-\gamma^2}.
\end{align}

From \eqref{nonconvex:vkLya-sm4_quantization}, \eqref{nonconvex:vkLya1.1_quantization}, and \eqref{nonconvex:vkLya3_determin}, we have
\begin{align*}
\sum_{k=0}^{T}\mathbf{E}_{\mathcal{C}}[\|\bsx_k\|^2_{\bsK}+\|\bar{\bsg}_{k}^0\|^2]
\le\frac{U_0+\frac{c_{8}}{1-\gamma^2}}
{\min\{\eta(\epsilon_1-\eta\tilde{\epsilon}_2),~\frac{\eta}{4}\}},
\end{align*}
which yields \eqref{nonconvex:thm-sm-equ1_quantization}.

From \eqref{nonconvex:vkLya-sm4_quantization}, \eqref{nonconvex:vkLya1.1_quantization}, and \eqref{nonconvex:vkLya3.2_determin}, we have
\begin{align*}
\mathbf{E}_{\mathcal{C}}[n(f(\bar{x}_T)-f^*)]\le U_T\le U_0+\frac{c_{8}}{1-\gamma^2},
\end{align*}
which yields \eqref{nonconvex:thm-sm-equ2_quantization} holds.

\subsection{Proof of Theorem~\ref{nonconvex:thm-ft_quantization}}\label{nonconvex:proof-thm-ft_quantization}
In this proof, in addition to the notations used in the proofs of Theorems~\ref{nonconvex:thm-sm}--\ref{nonconvex:thm-ft} and \ref{nonconvex:thm-sm_quantization}, we also denote
\begin{align*}
\tilde{\epsilon}&\in(0,\min\{\tilde{\epsilon}_0,~1-\gamma^2\}),
~\tilde{\epsilon}_0=\frac{\tilde{\epsilon}_{12}}{\epsilon_{11}},\\
\tilde{\epsilon}_{12}&=\eta\min\Big\{\epsilon_1-\eta\tilde{\epsilon}_2,
~\epsilon_3-\eta\tilde{\epsilon}_4,~\frac{\nu}{2}\Big\}.
\end{align*}

\noindent {\bf (i)}  We first show that $\tilde{\epsilon}\in(0,1)$.

From \eqref{nonconvex:vkLya1.1_quantization}--\eqref{nonconvex:vkLya1_quantization} and $\gamma\in(0,1)$, we have
\begin{align}\label{nonconvex:vkLya1.4_quantization}
\tilde{\epsilon}_{12}>0~\text{and}~\min\{\tilde{\epsilon}_0,~1-\gamma^2\}>0.
\end{align}

Noting that $\epsilon_1<\alpha\rho_2(L)/2$, $\tilde{\epsilon}_2>3\alpha^2\rho^2(L)$, and $\epsilon_{11}\ge\frac{1}{2}+\frac{\alpha}{\beta}\ge\frac{3}{2}$, we have
\begin{align}\label{nonconvex:vkLya-epsilon_quantization}
\tilde{\epsilon}<\tilde{\epsilon}_0=\frac{\tilde{\epsilon}_{12}}{\epsilon_{11}}
\le\frac{\eta(\epsilon_1-\eta\tilde{\epsilon}_2)}{\epsilon_{11}}
\le\frac{\epsilon_1^2}{4\tilde{\epsilon}_2\epsilon_{11}}<1.
\end{align}

\noindent {\bf (ii)}  We then show that \eqref{nonconvex:thm-ft-equ1_quantization} holds.

From \eqref{nonconvex:lemma:inequality-arithmetic-equ}, we have
\begin{align}\label{nonconvex:vkLya3.1_determin}
U_k\le\epsilon_{11}\hat{U}_k.
\end{align}

From \eqref{nonconvex:vkLya-sm2_determin}--\eqref{nonconvex:xminusxhat_quantization2}, \eqref{nonconvex:gg3}, and \eqref{nonconvex:vkLya1.4_quantization}--\eqref{nonconvex:vkLya3.1_determin}, we have
\begin{align*}
\mathbf{E}_{\mathcal{C}}[U_{k+1}]&\le \mathbf{E}_{\mathcal{C}}[U_{k}-\tilde{\epsilon}_{12}\hat{U}_k]+c_{8}\gamma^{2k}\nonumber\\
&
\le (1-\tilde{\epsilon}_0)\mathbf{E}_{\mathcal{C}}[U_{k}]+c_{8}\gamma^{2k}\nonumber\\
&\le(1-\tilde{\epsilon}_0)^{k+1}U_{0}
+\sum_{t=0}^{k}c_{8}(1-\tilde{\epsilon}_0)^t\gamma^{2(k-t)},
\end{align*}
which yields \eqref{nonconvex:thm-ft-equ1_quantization}.

\subsection{Proof of Theorem~\ref{nonconvex:thm-ft_determin}}\label{nonconvex:proof-thm-ft_determin}

In this proof, in addition to the notations used in the proofs of Theorems~\ref{nonconvex:thm-sm}--\ref{nonconvex:thm-ft} and \ref{nonconvex:thm-sm_quantization}--\ref{nonconvex:thm-ft_quantization}, we also denote
\begin{align*}
\hat{\kappa}_3&=\min\Big\{\frac{\hat{\epsilon}_1}{2\hat{\epsilon}_2},~\frac{\epsilon_5}{\epsilon_6},~\sqrt{\kappa_{15}}\Big\},
~\gamma\in[\max\{\sqrt{\kappa_{11}},~\sqrt{\kappa_{12}}\},1)\\
s_0&\ge\max\{\sqrt{\kappa_8/\kappa_7},~\max_{i\in[n]}\|x_{i,0}\|\},
~\kappa_{7}>\frac{\kappa_{10}(1-\varphi)^2}{\kappa_9},\\
\kappa_8&=\frac{1}{2}\|\bm{x}_0 \|^2_{\bsK}+\frac{1}{2}\Big\|\bsv_0
+\frac{1}{\beta}\bsg_0^0\Big\|^2_{\frac{\alpha+\beta}{\beta}\bsP}+\bsx_0^\top\bsK\bsP\Big(\bm{v}_0+\frac{1}{\beta}\bsg_0^0\Big)
+\frac{1}{2\nu}\|\bar{\bsg}_0^0\|^2,\\
\kappa_9&=\min\Big\{\frac{\hat{\epsilon}_1}{2\epsilon_{11}},~\frac{\nu}{2\epsilon_{11}}\Big\},
~\kappa_{10}=\Big(\tilde{\epsilon}_8+\frac{c_7\hat{\epsilon}_1}{\hat{\epsilon}_2}\Big)n\tilde{d}^2,\\
\kappa_{11}&=1-\eta\kappa_9+\frac{\eta\kappa_{10}(1-\varphi)^2}{\kappa_7},\\
\kappa_{12}&=(1+\varphi+\eta^2\kappa_{13})(1-\varphi)^2
+\frac{\eta^2\kappa_{14}}{\epsilon_{10}}\kappa_7,\\
\kappa_{13}&=4(1+\varphi^{-1})\hat{d}^2n\tilde{d}^2\alpha^2\rho^2(L),\\
\kappa_{14}&=4(1+\varphi^{-1})\hat{d}^2(\alpha^2\rho^2(L)+L_f^2),\\
\kappa_{15}&=\frac{(\varphi+\varphi^2-\varphi^3)\epsilon_{10}}{\kappa_{13}(1-\varphi)^2\epsilon_{10}
+\kappa_{14}\kappa_7},\\
\hat{\epsilon}_0&=\frac{\hat{\epsilon}_{12}}{\epsilon_{11}},~
\hat{\epsilon}_1=\min\{\epsilon_1,~\epsilon_3\},
~\hat{\epsilon}_2=\max\{\tilde{\epsilon}_2,~\tilde{\epsilon}_4\},\\
\hat{\epsilon}_{12}&=\eta\min\Big\{\hat{\epsilon}_1-\eta\hat{\epsilon}_2,~\frac{\nu}{2}\Big\},
~c_{9}=\eta(\tilde{\epsilon}_8+2c_7\eta)n\tilde{d}^2.
\end{align*}

\noindent {\bf (i)} We first show that $\kappa_{11},~\kappa_{12}\in(0,1)$.

From \eqref{nonconvex:beta1}--\eqref{nonconvex:beta3} and $0<\eta\le\hat{\kappa}_3\le\frac{\hat{\epsilon}_1}{2\hat{\epsilon}_2}$,  we have
\begin{subequations}\label{nonconvex:vkLya1.4_determin}
\begin{align}
&\hat{\epsilon}_{12}\ge\eta\min\Big\{\frac{\hat{\epsilon}_1}{2},
~\frac{\nu}{2}\Big\}>0,\label{nonconvex:vkLya1.4_determin1}\\
&\hat{\epsilon}_0=\frac{\hat{\epsilon}_{12}}{\epsilon_{11}}\ge\eta\kappa_9>0,
\label{nonconvex:vkLya1.4_determin2}\\
&0<c_9\le\kappa_{10}.\label{nonconvex:vkLya1.4_determin3}
\end{align}
\end{subequations}
From \eqref{nonconvex:vkLya-epsilon_quantization}, we have
\begin{align}\label{nonconvex:vkLya-epsilon_determin}
\eta\kappa_9\le\hat{\epsilon}_0\le\tilde{\epsilon}_0<1.
\end{align}

From \eqref{nonconvex:vkLya1.4_determin2} and $\kappa_{7}>\frac{\kappa_{10}(1-\varphi)^2}{\kappa_9}$, we have
\begin{align}\label{nonconvex:kappa11upp}
\kappa_{7}>0~\text{and}~\kappa_{11}=1-\eta\kappa_9+\frac{\eta\kappa_{10}(1-\varphi)^2}{\kappa_7}<1.
\end{align}
From \eqref{nonconvex:vkLya1.4_determin3}, \eqref{nonconvex:vkLya-epsilon_determin}, and $\kappa_{7}>0$, we have
\begin{align}\label{nonconvex:kappa11low}
\kappa_{11}>0.
\end{align}

From $\kappa_{7}>0$, $\varphi\in(0,1)$, and $\eta<\sqrt{\kappa_{15}}$, we have
\begin{align}\label{nonconvex:kappa12uppandlow}
0<\kappa_{12}&=(1+\varphi+\eta^2\kappa_{13})(1-\varphi)^2
+\frac{\eta^2\kappa_{14}\kappa_7}{\epsilon_{10}}\nonumber\\
&=1-(\varphi+\varphi^2-\varphi^3)+\frac{\eta^2(\kappa_{13}(1-\varphi)^2\epsilon_{10}
+\kappa_{14}\kappa_7)}{\epsilon_{10}}\nonumber\\
&<1.
\end{align}

\noindent {\bf (ii)} We next show that \eqref{nonconvex:thm-ft-equ1_determin} holds.

From \eqref{nonconvex:vkLya-sm2_determin}, \eqref{nonconvex:xminusxhat_determin}, \eqref{nonconvex:gg3},  \eqref{nonconvex:vkLya3.1_determin}, and \eqref{nonconvex:vkLya1.4_determin},  we have
\begin{align}\label{nonconvex:vkLya4_determin}
U_{k+1}&\le U_{k}-\hat{\epsilon}_{12}\hat{U}_k+c_{9}\max_{i\in[n]}\|x_{i,k}-\hat{x}_{i,k}\|^2_p\nonumber\\
&\le (1-\hat{\epsilon}_0)U_{k}+c_{9}\max_{i\in[n]}\|x_{i,k}-\hat{x}_{i,k}\|^2_p\nonumber\\
&\le (1-\eta\kappa_9)U_{k}+\eta\kappa_{10}\max_{i\in[n]}\|x_{i,k}-\hat{x}_{i,k}\|^2_p.
\end{align}

We have
\begin{align}\label{nonconvex:xminush_compress_determin}
\|x_{i,k+1}-\hat{x}_{i,k}\|^2_p&=\|x_{i,k+1}-x_{i,k}+x_{i,k}-\hat{x}_{i,k}\|^2_p\nonumber\\
&\le(\|x_{i,k+1}-x_{i,k}\|_p+\|x_{i,k}-\hat{x}_{i,k}\|_p)^2\nonumber\\
&\le(1+\varphi^{-1})\|x_{i,k+1}-x_{i,k}\|_p^2+(1+\varphi)\|x_{i,k}-\hat{x}_{i,k}\|_p^2\nonumber\\
&\le(1+\varphi^{-1})\hat{d}^2\|x_{i,k+1}-x_{i,k}\|^2+(1+\varphi)\|x_{i,k}-\hat{x}_{i,k}\|_p^2\nonumber\\
&\le(1+\varphi^{-1})\hat{d}^2\|\bsx_{k+1}-\bsx_{k}\|^2+(1+\varphi)\|x_{i,k}-\hat{x}_{i,k}\|_p^2,
\end{align}
where the first inequality holds due to the Minkowski inequality; the second inequality holds due to \eqref{nonconvex:lemma:inequality-arithmetic-equ} and $\varphi>0$; the third inequality holds due to Lemma~\ref{nonconvex:lemma:pnorm}.

We have
\begin{align}\label{nonconvex:xminush_compress_determin2}
\max_{i\in[n]}\|x_{i,k+1}-\hat{x}_{i,k}\|^2_p
&\le(1+\varphi+\eta^2\kappa_{13})
\max_{i\in[n]}\|x_{i,k}-\hat{x}_{i,k}\|_p^2+\eta^2\kappa_{14}\Big(\|\bsx_k\|^2_{\bsK}
+\Big\|\bm{v}_k+\frac{1}{\beta}\bsg_k^0\Big\|^2_{\bsP}\Big)\nonumber\\
&\le(1+\varphi+\eta^2\kappa_{13})
\max_{i\in[n]}\|x_{i,k}-\hat{x}_{i,k}\|_p^2+\frac{\eta^2\kappa_{14}}{\epsilon_{10}}U_{k},
\end{align}
where the first inequality holds due to \eqref{nonconvex:xkoneminusx}, \eqref{nonconvex:xminusxhat_determin}, and \eqref{nonconvex:xminush_compress_determin}; and the second inequality holds due to \eqref{nonconvex:vkLya3.2_determin}.

In the following, we use mathematical induction to prove
\begin{align}\label{nonconvex:Ukandx}
U_k\le \kappa_7s^2_k~\text{and}~\max_{i\in[n]}\|x_{i,k}-\hat{x}_{i,k-1}\|_p^2\le s_k^2.
\end{align}

From \eqref{nonconvex:gg3} and $s_0\ge\sqrt{\kappa_8/\kappa_7}$, we have
\begin{align}\label{nonconvex:UkandxU0}
U_0\le\kappa_8\le \kappa_7s_0^2.
\end{align}
From $s_0\ge \max_{i\in[n]}\|x_{i,0}\|$ and $\hat{x}_{i,-1}={\bf 0}_d$, we have
\begin{align}\label{nonconvex:Ukandxx0}
\max_{i\in[n]}\|x_{i,0}-\hat{x}_{i,-1}\|_p^2\le s_0^2.
\end{align}

Therefore, from \eqref{nonconvex:UkandxU0} and \eqref{nonconvex:Ukandxx0}, we know that \eqref{nonconvex:Ukandx} holds at $k=0$. Suppose that \eqref{nonconvex:Ukandx} holds at $k$. We next show that \eqref{nonconvex:Ukandx} holds at $k+1$.

We have
\begin{align}\label{nonconvex:xminusxhat_determin2}
\|x_{i,k}-\hat{x}_{i,k}\|_p
&=\|x_{i,k}-\hat{x}_{i,k-1}-s_k\mathcal{C}((x_{i,k}-\hat{x}_{i,k-1})/s_k)\|_p\nonumber\\
&=s_k\|(x_{i,k}-\hat{x}_{i,k-1})/s_k-\mathcal{C}((x_{i,k}-\hat{x}_{i,k-1})/s_k)\|_p\nonumber\\
&\le (1-\varphi)s_k,
\end{align}
where the first equality holds due to \eqref{nonconvex:kia-algo-dc-xhat_determin} and \eqref{nonconvex:kia-algo-dc-q_determin}; and the inequality holds due to \eqref{nonconvex:Ukandx} and \eqref{nonconvex:ass:compression_equ_determin}.

We have
\begin{align}\label{nonconvex:UkandxU1}
U_{k+1}&\le (1-\eta\kappa_9)\kappa_7s^2_k+\eta\kappa_{10}(1-\varphi)^2s_k^2
=\frac{\kappa_{11}}{\gamma^2}\kappa_7s^2_{k+1}\le\kappa_7s^2_{k+1},
\end{align}
where the first inequality holds due to \eqref{nonconvex:vkLya4_determin}, \eqref{nonconvex:Ukandx}, and \eqref{nonconvex:xminusxhat_determin2}; and the last inequality holds due to $\gamma\ge\sqrt{\kappa_{11}}$.

We have
\begin{align}\label{nonconvex:Ukandxx1}
\max_{i\in[n]}\|x_{i,k+1}-\hat{x}_{i,k}\|^2_p
&\le(1+\varphi+\eta^2\kappa_{13})(1-\varphi)^2s_k^2+\frac{\eta^2\kappa_{14}}{\epsilon_{10}}\kappa_7s^2_k
=\frac{\kappa_{12}}{\gamma^2}s^2_{k+1}\le s^2_{k+1},
\end{align}
where the first inequality holds due to \eqref{nonconvex:xminush_compress_determin2}, \eqref{nonconvex:Ukandx}, and \eqref{nonconvex:xminusxhat_determin2}; and the last inequality holds due to $\gamma\ge\sqrt{\kappa_{12}}$.

Therefore, from \eqref{nonconvex:UkandxU1} and \eqref{nonconvex:Ukandxx1}, we know that \eqref{nonconvex:Ukandx} holds at $k+1$. Finally, by mathematical induction, we know that \eqref{nonconvex:Ukandx} holds for any $k\in\mathbb{N}_0$. Hence, \eqref{nonconvex:thm-ft-equ1_determin} holds.

\bibliographystyle{IEEEtran}
\bibliography{refextra}








\end{document}